\documentclass[10pt,twoside]{amsart}

\usepackage{csquotes}

\usepackage[margin=1in]{geometry}

\usepackage{amsthm,amsfonts,amssymb,amsmath,latexsym,verbatim,amscd,mathrsfs,dsfont}
\newtheorem{thm}{\bf Theorem}[section]
\newtheorem{prop}[thm]{\bf Proposition}
\newtheorem{lemma}[thm]{\bf Lemma}
\newtheorem*{Def}{\bf Definition}
\newtheorem{rmk}[thm]{\bf Remark}

\newtheorem{cor}[thm]{\bf Corollary}

\def\R{\mathbb R}
\def\N{\mathbb N}
\def\d{\partial}
\def\ep{\varepsilon}
\renewcommand\div{{\rm div}\,}
\def\wt{\widetilde}
\def\cC{\mathcal C}
\def\sC{\mathscr C}
\def\sE{\mathscr E}

\def\cD{\mathcal D}
\def\sH{\mathscr H}
\def\sK{\mathscr K}

\def\cP{\mathcal P}
\def\cT{\mathcal T}
\def\cX{\mathcal X} 

\def\tDelta{\widetilde{\Delta}}
\numberwithin{equation}{section}
\usepackage{cleveref}

\begin{document}
\title[Persistence of geometrical structures for the Boussinesq equations]{Global persistence of geometrical structures for the Boussinesq equation with no diffusion}
\author{Rapha\"el DANCHIN and Xin ZHANG}
\email{raphael.danchin@u-pec.fr and xinzhang@univ-paris-est.fr}
\address{Universit\'{e} Paris-Est,  LAMA (UMR 8050), UPEMLV, UPEC, CNRS,  61 avenue du G\'en\'eral de Gaulle, 94010 Cr\'eteil Cedex 10. }

\begin{abstract}  Here we investigate the so-called \emph{temperature patch problem}  for the incompressible Boussinesq system with partial viscosity, in the whole space $\R^N (N \geq 2 ),$ where the initial temperature is the characteristic function  of some simply connected domain with  $C^{1,\ep}$ H\"older regularity.  Although recent results  in \cite{AH2007, DanP2008a} ensure that an initially $C^1$ patch persists through the evolution,  whether higher regularity is preserved   has remained an open question. In the present paper, we give  a positive answer to that issue globally in time, in the 2-D case for large initial data and  in the higher dimension case for small initial data.   
\end{abstract}
\maketitle

\section{Introduction}
This paper is devoted to  the temperature patch problem for the following  incompressible Boussinesq system with partial viscosity: \begin{equation}\tag{$B_{\nu,N}$}\label{eq:BsnqN}
\left\{
\begin{array}{l}
\d_t \theta +u\cdot \nabla\theta = 0, \\
\d_t u +u\cdot \nabla u-\nu \Delta u + \nabla \Pi = \theta e_N,\\
\div u = 0,\\
(\theta, u)|_{t=0} = (\theta_0 , u_0).
\end{array}
\right.
\end{equation}
Above,  $e_N=(0,\cdots,0,1)$ stands for the unit vertical vector in $\R^N$ with $N\geq 2.$ The unknowns are the scalar function $\theta$ (the temperature), the velocity field $u=(u^1, u^2,...,u^N)$ and the pressure $\Pi$, depending on the time variable $t\geq0$ and on the space variable $x\in\R^N.$ We assume the viscosity $\nu$ to be  a positive constant.
 
  The above Boussinesq system is a toy model for describing the convection phenomenon in viscous incompressible flows, and arises in simplified models for geophysics (see e.g. \cite{Ped1987}).  
A number of  works are dedicated to the  global well-posedness issue of $(B_{\nu,2})$ (see e.g. \cite{AH2007,Chae2006,DanP2008a,HK2007,HouL2005}). In particular, R. Danchin and M. Paicu  proved in \cite{DanP2008a}  (see also \cite{He2012}) that $(B_{\nu,2})$ has a unique global  solution $(\theta,u)$ such that
$\theta \in C\big(\R_{+}; L^2(\R^2)\big)$ and
\footnote{The notation $\widetilde{L}^1_{loc}(\R_{+}; H^2)$ designates a (close) superspace of  ${L}^1_{loc}(\R_{+}; H^2)$, see  \eqref{spc:tilde_H_C}.} 
\begin{equation}\label{spc:Leray_Bsnq2}
 u\in \Bigl(C\big(\R_{+}; L^2(\R^2)\big) \cap L^{2}_{loc}\big(\R_{+};H^1(\R^2)\big) \cap \widetilde{L}^1_{loc}\big(\R_{+}; H^2(\R^2)\big)\Bigr)^2,
\end{equation}
whenever the initial data $(\theta_0, u_0)$ are  in $\big(L^{2}(\R^2)\big)^3$ and satisfy  $\div u_0 =0.$   Additionally, the following energy equality is fulfilled for all $t\geq0$:
\begin{equation*}
\|u(t)\|^{2}_{L^{2}(\R^2)} + 2\nu \int^{t}_{0} \|\nabla u\|^2_{L^2(\R^2)} \,dt' = \|u_0\|^2_{L^2(\R^2)}
 + 2\int^{t}_{0} \int_{\R^2} \theta u_2 \,dx \,dt'.
\end{equation*}
Global well-posedness results are also available in dimension $N\geq3,$  but exactly as for the standard Navier-Stokes equations, the initial data have to satisfy a suitable smallness condition,  see for instance \cite{ DanP2008b, DanP2008a}.
\medbreak

To better explain the main motivation of our work, which is the \emph{temperature patch problem}, let us  assume that  $N=2$ for a while. Based on the aforementioned well-posedness result,  one may consider for $\theta_0$ the characteristic function of some simply connected bounded domain $\mathcal{D}_0$ of the plane. Given that $\theta$ is just advected by the velocity field $u,$  we expect to have $\theta(t,x) = \mathds{1}_{\mathcal{D}_t}(x)$ for all $t\geq0$, where $\mathcal{D}_t := \psi_{u} (t,\mathcal{D}_0)$ and $\psi_{u}$  stands for the flow associated to  $u$, that is 
to say the solution to the following (integrated) ordinary differential equation: 
\begin{equation}\label{eq:flow}
\psi_{u}(t,x): = x+ \int_{0}^{t} u(t', \psi_{u}(t',x))\,dt'.
\end{equation}
If  the regularity of $u$ is given by \eqref{spc:Leray_Bsnq2}, then it has been proved by J.-Y. Chemin and N. Lerner in \cite{CheL1995} that \eqref{eq:flow} has a unique solution, which is  in $\big(C(\R_+;C^{0,1-\eta})\big)^2$ for any $\eta>0.$  
Now, if  we add a bit more regularity on the initial data, for instance\footnote{See the definition of Besov spaces in the next section.} $u_{0} \in (B^{0}_{2,1})^2$ and $\theta_{0} \in B^{0}_{2,1}$  then, according to \cite{AH2007}, 
all the entries of $\nabla u$ are in $L^1_{loc}(\R_+;C_b),$  and  the flow $\psi_u(t,\cdot)$ is thus $C^1.$ Consequently, the   $C^1$ regularity of the temperature patch is preserved for all time.  

Then a natural question arises: what  if we start with a  $C^{1,\ep}$ H\"older domain $\mathcal{D}_0$ with $\ep\in]0,1[$~? Our concern  has some similarity with the celebrated  vortex patch problem for the  2-D incompressible Euler equations.
In that case, it has been proved (see e.g. \cite{Che93,Che1998} and the references therein) that  the  $C^{1,\ep}$ regularity of the patch of the vorticity persists for all time. Proving that in our framework, too, the H\"older regularity of  $\mathcal{D}_t$ is conserved is the main purpose of the present paper. Just like for the vortex patch problem for Euler equations,  our result  will come up as a consequence of a much more general property  of global-in-time   persistence of  \emph{striated regularity}, a definition that originates from the work of J.-Y. Chemin in  \cite{Che1988}.
\medbreak

Before stating our main results, we need to  introduce some notation, and to clarify what striated regularity is. 
Assume that\footnote{We adopt Einstein summation convention  in the whole text: summation is taken  with respect to the repeated indices, whenever they occur both as a subscript and a superscript.}
 $X = X^{k}(x) \d_k$ is some vector field acting on  functions in $C^1(\R^N;\R).$  As usual, vector fields are identified  with vector valued functions from $\R^N$ to $\R^N,$ 
and $\d_X f$ stands for the \emph{directional derivative} of $f \in C^1(\R^N;\R)$ along the vector field $X$, namely  
$$\d_X f :=  X^k \d_k f=X\cdot \nabla f.$$

The evolution $X_t(x):= X(t,x)$  of any continuous  initial vector field $X_0$
along the flow of $u$ is defined by:
$$ X(t,x) := (\d_{X_0} \psi_u)\big(\psi^{-1}_{u}(t,x)\big).$$
In the $C^1$ case,  combining the chain rule and the definition of the flow in \eqref{eq:flow} implies  that  $X$ satisfies the transport  equation
\footnote{Omitting  the index $t$ in $X_t$ for notational simplicity.}
\begin{equation}\label{eq:X}
\left\{
\begin{array}{l}
\d_t X +u\cdot \nabla X = \d_{X} u, \\
 X|_{t=0}=X_0. \\
\end{array}
\right.
\end{equation}
Applying operator  $\div$ to \eqref{eq:X} and remembering that $\div u =0$, we 
obtain in addition
\begin{equation}\label{eq:divX}
\left\{
\begin{array}{l}
 \d_t \div X + u\cdot \nabla \div X = 0,\\
 \div X |_{t=0}=\div X_0. \\
\end{array}
\right.
\end{equation}
This implies that the divergence-free  property is  conserved through the evolution. 
\medbreak
As we will see in Section \ref{sect:5}, the temperature patch problem is closely related to the conservation of  H\"older  regularity $C^{0,\ep}$ for $X.$ According to the classical theory of transport equations, if $u$ is Lipschitz with respect to the space variable (a condition that will be ensured if the data of $(B_{\nu,N})$ have  critical Besov regularity), then conservation of $C^{0,\ep}$ regularity for $X$ is equivalent to the fact that all the components of  $\d_Xu$ have the regularity $C^{0,\ep}$ 
with respect to the space variable.
\medbreak
In  the 2-D case, it is natural to recast the regularity of $u$ along the vector field $X$ in terms of the vorticity
$\omega := \d_1 u^2-\d_2 u^1$ as the simple transport-diffusion equation is fulfilled:
\begin{equation}\label{eq:voritcity_Bsnq2}
\left\{
\begin{array}{l}
\d_t \omega +u\cdot \nabla \omega -\nu \Delta \omega= \d_1 \theta, \\
 \omega|_{t=0}=\omega_0,
\end{array}
\right.
\end{equation}
and as it is known (see e.g. \cite{BCD2011}, Chap. 7) that\footnote{In all the paper, we agree that $A \lesssim B$ means  $A \leq C B$ for some harmless constant $C.$}:
\begin{equation}\label{es:dXu_divXomega1}
\|\d_Xu\|_{\sC^\ep}\lesssim \|\nabla u\|_{L^\infty}\|X\|_{\sC^\ep}+\|\div(X\omega)\|_{\sC^{\ep-1}},
\end{equation}
where  for any real number $s,$ we 
denote\footnote{Recall that  $\sC ^{k+\ep}$ coincides with the standard H\"older space $C^{k,\ep}$ whenever $k\in\N$ and $\ep\in]0,1[.$}
$\sC^{s} \equiv \sC ^{s}(\R^N) := B^{s}_{\infty,\infty}(\R^N).$
\medbreak
Now,  applying operators $\d_X$ and $\div(X\cdot)$  to the temperature and vorticity equations,  respectively, we get the following system for $\big(\d_X\theta,\div(X\omega)big)$:
\begin{equation}\label{eq:dBsnq2}
\left\{
\begin{aligned}
&\d_t \d_X \theta + u\cdot \nabla \d_X \theta = 0,\\
&\d_t\div(X\omega)+u\cdot \nabla \div (X\omega) -\nu \Delta \div(X\omega) =f,
%\\&\big(\d_{X} \theta, \div(X \omega)\big)|_{t=0} = \big(\d_{X_0} \theta_0, \div ({X_0} \omega_0)\big),
\end{aligned}
\right.
\end{equation}
with $\displaystyle f:=\nu \div \big(X\Delta \omega - \Delta(X\omega)\big)+ \div (X \d_1 \theta).$
\medbreak

Let us recap. Roughly speaking, to propagate the $\sC^\ep$ regularity of $X,$ it is sufficient to control 
$\d_Xu$ in  $\big({L}^1_{loc}(\R_+;\sC^{\ep})\big)^2,$ and this may be achieved, thanks to \eqref{es:dXu_divXomega1}, 
if bounding the distribution $\div(X\omega)$ in ${L}^1_{loc}(\R_+;\sC^{\ep-1}).$ Then by taking 
advantage of smoothing properties of the heat flow,  this latter information may be obtained through a bound of   $f$ in  the very negative space ${L}^1_{loc}(\R_+;\sC^{\ep-3}),$ if assuming that $\div(X_0\omega_0)\in\sC^{\ep-3}.$
Staring at  the expression of $f,$ we thus need to  bound  $\d_X\theta$ in $L^\infty_{loc}(\R_+;\sC^{\ep-2}).$
As no gain of regularity may be expected from the transport equation satisfied by $\d_X\theta,$  we have to assume initially that $\d_{X_0}\theta_0\in\sC^{\ep-2}.$ 
 This motivates the following statement which is our main result of propagation of striated regularity 
 in the 2-D case
 \footnote{To fully benefit from the smoothing properties of the heat flow in the endpoint case, one has to work in (close) superspaces  of ${L}^1_{loc}(\R_+;\sC^{s})$ denoted by $\widetilde {L}^1_{loc}(\R_+;\sC^{s})$ and defined in \eqref{spc:tilde_H_C}.}.
\begin{thm}\label{thm:dXBsnq2}
Suppose that $(\ep, q)\in ]0,1[\times ]1,\frac{2}{2-\ep}[$. Let $\theta_0$ be in  $B^{\frac{2}{q}-1}_{q,1}(\R^2)$ and $u_0$ be a divergence-free vector field in $\big( L^2(\R^2)\big)^2,$ with  vorticity $\omega_0 := \d_1 u^2_0-\d_2 u_0^1$ in $B^{\frac{2}{q}-2}_{q,1}(\R^2)$. Then there exists a unique global solution $(\theta, u)$ of System $(B_{\nu,2})$, 
such that
\begin{equation}\label{spc:thm_dXBsnq2}
(\theta, u, \omega) \in C (\R_{+}; B^{\frac{2}{q}-1}_{q,1}) \times\big( C(\R_+;L^2)\big)^2 
\times \big( C(\R_{+};B^{\frac{2}{q}-2}_{q,1}) \cap  L^1_{loc}(\R_{+};B^{\frac{2}{q}}_{q,1})\big).
\end{equation}
Furthermore, if we consider some   $X_0$ in  $\big(\sC^{\ep}(\R^2)\big)^2$ satisfying
$\d_{X_0}\theta_0 \in \sC^{\ep-2}(\R^2)$ and $\div(X_0 \omega_0) \in \sC^{\ep-3}(\R^2),$
then there exists a unique global solution
\footnote{If $E$ is a Banach space with predual $E^*$ then  $C_{w}(\R_+;E)$ stands for the set of measurable functions $h:\R_+\to E$ such that for all $\phi\in E^*,$ the function $t\mapsto\langle h(t),\phi\rangle_{E\times E^*}$
is continuous on $\R_+.$}
 $X\in C_w(\R_+;\sC^\ep)$ to \eqref{eq:X} and we have 
\begin{equation*}
\big(\d_X \theta, \div (X \omega)\big) \in  {C}_w (\R_{+}; \sC^{\ep-2}) \times \big({C}_{w}(\R_{+}; \sC^{\ep-3}) \cap \widetilde{L}^{1}_{loc}(\R_{+}; \sC^{\ep-1})\big).
\end{equation*}
Additionally, there is a constant $C_{0, \nu} $ depending only on  the initial data and viscosity constant such that for any $t\geq0$,
\begin{equation*}
\|X\|_{L^{\infty}_{t}(\sC^{\ep})} \leq C_{0,\nu} \exp \Big(\exp \big(\exp ( C_{0, \nu}t^4)\big) \Big).
\end{equation*}
\end{thm}

A few comments are in order:
\begin{itemize}
\item The functional space $ B^{\frac{2}{q}-1}_{q,1}(\R^2)$ for $\theta_0$ is large enough to contain the characteristic function of any bounded $C^1$ domain. This will be  needed  to investigate  the temperature patch problem later (see 
Corollary \ref{cor:dXBsnq2} and Section \ref{sect:5} below).
\item It may be seen by means of elementary paradifferential  calculus  that
if $(\ep, q)\in ]0,1[\times ]1,\frac{2}{2-\ep}[$ then $\div(X_0\omega_0)$ and  $\d_{X_0}\theta_0$
are distributions of $\sC^{-3}$ and $\sC^{-2},$ respectively, and that
the following (sharp) estimates are fulfilled:
$$
\|\div (X_0 \omega_0)\|_{\sC^{-3}}  \lesssim \|X_0\|_{\sC^{\ep}}\|\omega_0\|_{B^{\frac{2}{q}-2}_{q,1}}\quad\hbox{and}\quad
\|\d_{X_0} \theta_0\|_{\sC^{-2}}  \lesssim \|X_0\|_{\sC^{\ep}}\|\theta_0\|_{B^{\frac{2}{q}-1}_{q,1}}.
$$
 Therefore, our  striated regularity assumption on the initial data is indeed additional information.

\item The required  level of regularity is  much lower than in  the  (inviscid) 2-D vortex patch problem where  $\div(X_0 \omega_0) \in \sC^{\ep-1}(\R^2)$ is needed.  This is because  the  smoothing effect given by the  heat flow enables us to 
gain two derivatives with respect to the initial data. Another difference is that to tackle  the temperature patch problem, it is not necessary to consider a family of vector fields that does not degenerate on the whole $\R^2$ : just one suitably chosen vector-field that does not vanish in the neighborhood of the boundary of the patch is enough, as we shall see just below. 
\end{itemize}
\medbreak
Let us now go to the temperature patch problem in the  $2$-D case.
More precisely, consider a  $C^{1,\ep}$ simply connected bounded domain ${\mathcal D}_0$ of $\R^2$
(in other words  $\d\mathcal{D}_0$ is a $C^{1,\ep}$ Jordan curve on $\R^2$).  Let  $\mathcal{D}^{\star}_0$ be any bounded domain 
of $\R^2$ such that\footnote{We denote by $\overline{A}$ the closure of the subset $A$ in $\R^N,$ $N\geq2$.}  
\begin{equation}\label{as:dist}
\overline{\cD_0}\cap\overline{\cD^{\star}_0}=\emptyset. 
  \end{equation} 
Then the following result holds true. 

\begin{cor}\label{cor:dXBsnq2} Let $(M_1, M_2) $ be in $\R^2.$
Assume that $\theta_0 = M_1 \mathds{1}_{\mathcal{D}_0}$  
 and  that the vorticity  $\omega_0$ of $u_0$ may be decomposed
 into 
 \begin{equation}\label{as:vorticity}
 \omega_0 =M_2 \mathds{1}_{\mathcal{D}_0}-\wt\omega_0
 \end{equation}
 for some $\wt\omega_0\in L^r(\R^2)$ with  $r>1,$ supported in $ \overline{{\mathcal D}_0^\star}$
 and such that 
 \begin{equation}\label{as:mean}
 \int_{\R^2}\wt\omega_0(x)\,dx=M_2\,|{\mathcal{D}_0}|.
 \end{equation}
 Then there exists a unique solution $(\theta,u)$ to System $(B_{\nu,2}),$
 satisfying the properties listed in Theorem \ref{thm:dXBsnq2}. 
 Furthermore, we have  $\theta(t,\cdot) = M_{1}\mathds{1}_{\mathcal{D}_t}$  where 
 $\mathcal{D}(t) := \psi_{u} (t,\mathcal{D}_0)$ and $\d\mathcal{D}(t)$ remains a  $C^{1,\ep}$
 Jordan curve of $\R^2$  for all $t\geq0.$
\end{cor} 
Let us make some comments on that corollary.
\begin{itemize}
\item Hypothesis \eqref{as:mean} ensures that the initial vorticity is mean free, which is necessary to have $u_0\in \big(L^2(\R^2)\big)^2.$  As
a matter of fact, the more natural assumption $\omega_0=M_2 \mathds{1}_{\cD_0}$
would require our extending  Theorem \ref{thm:dXBsnq2} to infinite energy velocity fields, which introduces  additional technicalities. 
\item 
Assumption  \eqref{as:vorticity} on the vorticity may seem somewhat artificial as  it has no persistency whatsoever through the time evolution, 
even in the asymptotics $\nu\to0$ (in contrast with the slightly viscous
vortex patch problem, see \cite{Dan1997}). This is just to have a concrete
example of initial velocity for which one can give a positive answer to the temperature
patch problem. 
\item  Corollary \ref{cor:dXBsnq2} may be generalized to the case where the 
initial velocity $u_0\in \big(L^2(\R^2)\big)^2$ is such that $\omega_0\in B^{\frac2q-2}_{q,1}(\R^2)$ for some 
$1<q<\frac2{2-\ep}$ and satisfies $\div(X_0\omega_0)\in\sC^{\ep-3}(\R^2)$ for some 
vector field $X_0\in\sC^\ep(\R^2)$ that  does not vanish on $\d\mathcal{D}_0$ 
and is tangent to $\d\mathcal{D}_0.$ One just has to follow the proof that is  proposed in 
Section \ref{sect:5} to get this more general result.
\end{itemize}
\bigbreak
In space dimension $N \geq 3$, the  vorticity equation has an additional stretching term, and it is thus less natural to measure the striated regularity by means of $\div(X\Omega)$ with $\Omega$ denoting the matrix of $\hbox{curl}~u$  (even though we suspect that our 2-D approach is adaptable to the  high-dimensional case, like in \cite{GS-R1995}). We shall thus concentrate on 
the regularity of  $\d_X\theta$ and  $\d_Xu.$ An additional  (related) difficulty is that one cannot expect to prove global existence  for general large initial data, since   \eqref{eq:BsnqN} contains the standard incompressible Navier-Stokes equations as a particular case. Therefore we shall prescribe some smallness condition on the data (the same one as in \cite{DanP2008a}) to achieve a global statement. This leads to the following theorem:

\begin{thm}\label{thm:dXBsnqN}
Suppose that $N \geq 3$ and that $(\ep, p) \in]0,1[ \times ]N, \frac{N}{1-\ep}[$. Assume that $\theta_0$ is in 
$B^{0}_{N,1}(\R^N) \cap L^{\frac{N}{3}}(\R^N)$ and that the components of the
divergence-free vector field $u_0$ are in $B^{\frac{N}{p}-1}_{p,1}(\R^N)$
and in the weak Lebesgue space  $L^{N,\infty}(\R^N)$. If there exists a (small) positive constant $c$ independent of $p$ such that
\begin{equation*}
\|u_0\|_{L^{N,\infty} } + \nu^{-1} \|\theta_0\|_{L^{\frac{N}{3}}} \leq c\nu,
\end{equation*}
then Boussinesq system \eqref{eq:BsnqN} has a unique global solution
\begin{equation*}
(\theta, u, \nabla \Pi) \in C(\R_{+};B^{0}_{N,1}) \times \big(C(\R_{+};B^{\frac{N}{p}-1}_{p,1}) \cap L^1_{loc}(\R_{+};B^{\frac{N}{p}+1}_{p,1})\big)^N \times \big( L^1_{loc}(\R_{+};B^{\frac{N}{p}-1}_{p,1})\big)^N.
\end{equation*} 
\medbreak

Moreover, suppose that the vector field $X_0$ is in the space $\big(\widetilde{\sC}^{\ep}(\R^N)\big)^N$ defined by

\begin{equation*}
\big( \widetilde{\sC}^{\ep}(\R^N) \big)^N:= \{ Y \in \big(\sC^{\ep}(\R^N)\big)^N: \div Y \in \sC^{\ep}(\R^N)\},
\end{equation*}
and that the components of $(\d_{X_0}\theta_0, \d_{X_0}u_0)$ are in $\sC^{\ep-2}(\R^N).$ Then the System \eqref{eq:X} has a unique solution $X\in C_w(\R_+;\wt\sC^\ep),$  that satisfies  for all $t\geq0,$
\begin{equation*}
\|X\|_{L^{\infty}_{t}(\widetilde{\sC}^{\ep})} \leq C_{0,\nu} \exp \big( \exp ( C_{0, \nu}t) \big)
\end{equation*}
with some constant $C_{0, \nu} $ depending only on  the initial data and on $\nu.$

Furthermore,  the triplet $( \d_X \theta ,\d_X u, \d_X\nabla\Pi)$  belongs to  
\begin{equation*}
 C_w(\R_+; \sC^{\ep-2}) \times \big(C_w (\R_+; \sC^{\ep-2}) \cap \widetilde{L}^1_{loc}(\R_+;\sC^{\ep}) \big)^N \times \big(\widetilde{L}^1_{loc}(\R_+;\sC^{\ep-2})\big)^N.
\end{equation*}
\end{thm}

As in the 2-D case, the above result will enable us to solve the temperature patch problem. 
Before giving the exact statement, let  us recall what a  $C^{1,\ep}$ domain is in dimension $N\geq2.$ 
\begin{Def}\label{def:domain}
A simply connected bounded domain $\cD \subset \R^N$ is of class $C^{1,\ep}$ if  its boundary  $\d \cD$ is some compact hypersurface of class $C^{1,\ep}$. 
\end{Def}
Fix some domain  $\cD_0$  of class $C^{1,\ep}$ and further consider
another $C^1$ simple bounded domain $J_0$ such that
$\overline{\cD_0} \subset J_0.$   Then we have the following statement\footnote{Like in the 2-D case, it goes without saying that much 
more general initial velocities may be considered}:
\begin{cor} \label{cor:striated_BsnqN} 
Let $N\geq3$ and $(m_1, m_2)$ be a pair of sufficiently small constants. Assume that $\theta_0 = m_1 \mathds{1}_{\mathcal{D}_0}$  and that  the initial vorticity 
$\Omega_0 := \hbox{curl}~u_0$, i.e. $(\Omega_0)^{i}_{j} := \d_{j} u_0^i -\d_{i} u_0^j $ for any $i,j=1,...,N,$
satisfies  $\Omega_0:= m_2\mathds{1}_{J_0} \mathbb{A}_0$ where   $\mathbb{A}_0$ stands for 
the  anti-symmetric matrix defined by  $(\mathbb{A}_0)^i_j=1$ for  $i<j$.  

Then $\theta(t,\cdot) = m_1 \mathds{1}_{\mathcal{D}_t}$  where $\mathcal{D}(t) := \psi_{u} (t,\mathcal{D}_0)$ and $\mathcal{D}(t)$ remains a simply connected domain of class $C^{1,\ep}$, for any $t \geq 0$. 
\end{cor}

The rest of the paper unfolds as follows. In the next section,  we shortly introduce Besov spaces and present some linear or nonlinear estimates, which will be needed to achieve our  results. Then the proofs of main theorems for the propagation of striated regularity will be revealed in Section \ref{sect:striated_Bsnq2} (for 2-D case) and Section \ref{sect:thm_dXBsnqN} (for N-D case). Section
\ref{sect:5} is devoted to  the temperature patch problems. Some  technical commutator bounds are  proved in the Appendix.

%%%%%%%%%%%%%%%%%%%%%%%%%%%%%%%%%%%%%

\section{Basic notations and linear estimates}\label{sect:2}

We here introduce definitions and notations that are  used throughout the text, 
and recall some properties of Besov spaces and transport or transport-diffusion equations.

Let us begin with the definition of the \emph{nonhomogeneous Littlewood-Paley decomposition}  (for more details see \cite{BCD2011}, Chap. 2). Set $\mathcal{B}:= \{ \xi \in \R^N :|\xi| \leq {4}/{3} \}$ and $\mathcal{C}:= \{ \xi \in \R^N : {3}/{4} \leq |\xi| \leq {8}/{3} \}$. We fix two  smooth radial functions $\chi$ and $\varphi,$ supported in $\mathcal{B}$ and $\mathcal{C}$, respectively, and such that
\begin{equation*}
\chi(\xi) + \sum_{j\geq 0} \varphi(2^{-j}\xi)=1, ~~~\forall ~\xi \in \R^N.
\end{equation*}
We then introduce the Fourier multipliers $\Delta_{-1}:= \chi(D)$ and $\Delta_j := \varphi(2^{-j}D)$ with $j \geq 0$ 
(the so-called \emph{nonhomogeneous dyadic blocks})  and  the low frequency cut-off operator  $$S_j := \sum_{j' \leq j-1} \Delta_{j'}.$$

 With those notations, the nonhomogeneous Besov space  $B^s_{p,r}(\R^N)$ may be defined~by
\begin{equation*}
B^{s}_{p,r}(\R^N):= \{ u \in \mathcal{S}'(\R^N): \|u\|_{B^s_{p,r}} :=\big\|2^{js}\|\Delta_j u \|_{L^p} \big\|_{\ell^r(\N\cup\{-1\}) } <\infty\},
\end{equation*}
for   $(s,p,r) \in \R \times [1,\infty]^2$.
\medbreak

Let us next introduce the following \emph{paraproduct} and \emph{remainder} operators:
$$T_u v := \sum_{j\geq-1} S_{j-N_0}u \Delta_j v\quad\hbox{and}\quad 
R(u,v) \equiv \sum_{j\geq-1} \Delta_{j} u \tDelta_{j} v := \sum_{\genfrac{}{}{0pt}{}{j\geq-1}{|j-k|\leq N_0}} \Delta_{j} u \Delta_{k} v,
$$
where $N_0$ stands for some large enough (fixed) integer.
\medbreak

The following  decomposition, first introduced by J.-M. Bony in \cite{Bo1981}:
\begin{equation}\label{eq:bony}
uv = T_u v + T_v u + R(u,v),
\end{equation} 
  holds true whenever  the product of the two tempered distribution $u$ and $v$
is defined. It will play a   fundamental role   in our study.
\medbreak
 
Bilinear operators $R$ and $T$ possess continuity properties in a number of functional spaces
(see e.g. Chap. 2 in  \cite{BCD2011}). We shall recall a few of them throughout the text,  when needed.
\medbreak

When   investigating evolutionary equations in Besov spaces and, in particular, parabolic type equations, it is natural to use
the following \emph{tilde Besov spaces} first introduced by J.-Y. Chemin  in \cite{Che1999}:
 for any $t\in ]0,\infty]$ and $(s,p,r,\rho) \in \R \times [1,\infty]^3$, we set
\begin{equation*}
\widetilde{L}^\rho_t \big(B^{s}_{p,r}(\R^N)\big):= \Bigl\{ u \in \mathcal{S}'(]0,t[ \times \R^N): \|u\|_{\widetilde{L}^\rho_t (B^{s}_{p,r})} :=\big\|2^{js}\|\Delta_j u\|_{L^\rho(]0,t[;L^p)} \big\|_{\ell^r } <\infty\Bigr\}\cdotp
\end{equation*}
In the particular case where $p=r=2$ (resp. $p=r=\infty$), $B^s_{p,r}$ coincides with the Sobolev space $H^s$ (resp. the generalized H\"older space
$\sC^s$),  and we shall
alternately denote 
\begin{equation}\label{spc:tilde_H_C}
\widetilde{L}^\rho_t \big(H^{s}(\R^N)\big):=\widetilde{L}^\rho_t \big( B^{s}_{2,2}(\R^N)\big)\quad\hbox{and}\quad
\widetilde{L}^\rho_t \big(\sC^{s}(\R^N)\big):=\widetilde{L}^\rho_t \big( B^{s}_{\infty,\infty}(\R^N)\big).
\end{equation}
Let us next state some a priori estimates for the transport and transport-diffusion equations in (nonhomogeneous) Besov spaces. 
\begin{prop}\label{prop:trans_diff} Assume that $v$ is a divergence free vector field. 
Let $(p,p_1,r,\rho,\rho_1) \in [1,\infty]^5$ and $s\in\R$ satisfy 
\begin{equation*}
 -1-\min\biggl(\frac{N}{p_1}, \frac N{p'}\biggr)<s< 1+ \min\biggl(\frac Np,\frac N{p_1}\biggr) ~~~\mbox{  and  }~~~ \rho_1 \leq \rho .
\end{equation*}
Let $f$ be a smooth solution of   the following transport-diffusion equation with diffusion parameter  $\nu \geq 0$:
\begin{equation*}\tag{$TD_{\nu}$}
\left\{
\begin{array}{l}
 \d_t f + \div(fv) - \nu \Delta f = g,\\
 f|_{t=0}=f_0. \\
\end{array}
\right.
\end{equation*}
Then there exists a constant C depending on $N,$ $p,$ $p_1$  and $s$ such that for all $t\geq0$,
\begin{multline}\label{es:transdiff}
\nu^{\frac{1}{\rho}} \|f\|_{\widetilde{L}^{\rho}_{t}(B^{s+\frac{2}{\rho}}_{p,r})} \leq C e^{C(1+\nu t)^{\frac{1}{\rho}}V_{p_1}(t)}\Big( (1+\nu t)^{\frac{1}{\rho}} \|f_0\|_{B^{s}_{p,r}} \\
 +(1+\nu t)^{1+\frac{1}{\rho}-\frac{1}{\rho_1}} \nu ^{\frac{1}{\rho_1}-1} \|g\|_{\widetilde{L}^{\rho_1}_{t}(B^{s-2+\frac{2}{\rho_1}}_{p,r})} \Big),
\end{multline}
where 
\begin{equation*}
V_{p_1} (t):= \int^t_0 \|\nabla v\|_{B^{\frac{N}{p_1}}_{p_1,\infty} \cap L^\infty}\,dt'.
\end{equation*}
In  the limit case $s=-1-\min\bigl(\frac{N}{p_1}, \frac{N}{p'}\bigr),$ one just need to refine 
 $ \|\nabla v\|_{B^{\frac{N}{p_1}}_{p_1,\infty} \cap L^\infty}$ by $\|\nabla v\|_{B^{\frac{N}{p_1}}_{p_1,1}}$ in the definition of $V_{p_1}.$
\end{prop}
\begin{proof}
The above statement has been proved in \cite{Dan2007} in the case $p\leq p_1.$
The generalization to the case $p > p_1$ is straightforward: 
in fact that restriction came  from the following 
commutator estimate\footnote{We here adopt the usual notation  $[A, B]$  for the commutator $AB-BA$.}:
\begin{equation*}
\big\|2^{js} \|[v\cdot \nabla, \Delta_j]f\|_{L^{p}} \big\|_{\ell^r} \lesssim \|\nabla v\|_{B^{\frac{N}{p_1}}_{p,1}}\|f\|_{B^{s}_{p,r}},
\end{equation*}
that has been proved only in the case $p \leq  p_1.$
\medbreak
To handle the case $p>p_1,$ we proceed exactly as in \cite{Dan2007}, 
decomposing  $R_j := [v\cdot \nabla, \Delta_j]f$
into  $R_j  =\sum_{i=1}^6 R_j^{i}$ with 
$$
\begin{array}{llllll}
R_j^{1}&:=&[T_{\tilde v^{k}},\Delta_{j}]\partial_{k}f,&\qquad
R_j^{2}&:=&T_{\partial_{k}\Delta_{j}f}\tilde v^{k},\\[1ex]
R_j^3&:=&-\Delta_{j}T_{\partial_{k}f}\tilde v^{k},&\qquad
 R_j^{4}&:=&\partial_kR(\tilde v^{k},\Delta_{j}f),\\[1ex]
 R_j^{5}&:=&-\partial_k\Delta_{j}R(\tilde v^{k},f),&\qquad
R_j^6&:=&[S_0v^k,\Delta_j]\partial_k f,
\end{array} 
$$ 
where  $\tilde{v}:= v - S_0 v.$
\medbreak
In  \cite{Dan2007}, Condition  $p \leq p_1 $ is used only when bounding $R^3_j.$
Now, if $p > p_1$ and $s < 1+\frac{N}{p}$ then combining standard continuity results for the paraproduct
with the embedding $B^{s-1}_{p,r}(\R^N)\hookrightarrow B^{s-1-\frac Np}_{\infty,r}(\R^N)$ implies that
$$
\big\| 2^{j(s+\frac{N}{p_1}-\frac{N}{p})}\|R_j^3\|_{L^p} \big\|_{\ell^r} \lesssim \|\nabla v\|_{B^{\frac{N}{p_1}}_{p_1,\infty}}\|\nabla f\|_{B^{s-1}_{p,r}},
$$
whence the desired inequality. 
\end{proof}

\begin{rmk} We shall often use the above proposition in the  particular case   $\nu =0$ and $(p,r,\rho,\rho_1) = (\infty,\infty,\infty,1).$
Then  Inequality \eqref{es:transdiff} reduces to
\begin{align*}
\|f\|_{L^{\infty}_{t}(\sC^{s})} \leq C e^{C V_{p_1}(t)}  &\big( \|f_0\|_{\sC^{s}} + \|g\|_{\widetilde{L}^{1}_{t}(\sC^{s})} \big)\quad\hbox{if}\quad
-1-\frac N{p_1}<s<1.
\end{align*}\end{rmk}

We finally recall a refinement of  Vishik's  estimates for the transport equation \cite{Vish1998}
obtained  by T. Hmidi and S. Keraani in \cite{HK2008}, which 
is the key to the study of long-time behavior of the solution in critical spaces for Boussinesq system \eqref{eq:BsnqN} (see  \cite{AH2007, DanP2008a}).
\begin{prop}\label{prop:trans_vishik}
Assume that $v$ is divergence-free and that $f$  satisfies the transport equation  $(TD_0)$. There exists a constant $C$ such that for all $(p,r) \in [1,\infty]^2$ and $t>0$, 
\begin{equation*}
\|f\|_{\widetilde{L}^{\infty}_{t}(B^{0}_{p,r})} \leq C \big( \|f_0\|_{B^{0}_{p,r}}+ \|g\|_{\widetilde{L}^{1}_{t}(B^{0}_{p,r})} \big) 
\biggl( 1+ \int_0^t\|\nabla v\|_{L^{\infty}}\,d\tau\biggr)\cdotp
\end{equation*}
\end{prop}

%%%%%%%%%%%%%%%%%%%%%%%%%%%%%%%%%

\section{Propagation of striated regularity in the 2D-case}\label{sect:striated_Bsnq2}

This section is devoted to the proof of  Theorem \ref{thm:dXBsnq2}. To simplify the computations, 
we shall first make the change of unknowns
\begin{equation}\label{change:Bsnq}
\wt{\theta}(t,x)=\nu^2\theta(\nu t,x),\quad \wt{u}(t,x)=\nu u(\nu t,x),\quad
\wt{ P}(t,x)=\nu^2 P(\nu t,x)
\end{equation}
so as to reduce the study to the case $\nu=1.$ 
\medbreak
Throughout  this section, we always assume that
\begin{equation}\label{cdt:striated_Bsnq2}
0<\ep<1, \quad q > 1  ~~\mbox{  and  }~~ \frac{\ep}{2} + \frac{1}{q} > 1,
\end{equation}
in accordance with the hypotheses of Theorem \ref{thm:dXBsnq2}.

\subsection{A priori estimates for the Lipschitz norm of the  velocity field}

Those estimates will be based on the following global existence theorem (see \cite{AH2007}).
\begin{thm}
Let $u_0$ be a divergence-free vector field belonging to the space $\big(L^2(\R^2) \cap B^{-1}_{\infty,1}(\R^2)\big)^2$ and 
let $\theta_0$ be in $B^{0}_{2,1}(\R^2)$. 
Then there exists a unique global solution $(u,\theta, \nabla\Pi)$ for System $(B_{1,2})$ such that 
\begin{equation}\label{spc:AH2007}
\begin{array}{c}
u \in \big( C(\R_{+};L^2 \cap B^{-1}_{\infty,1}) \cap L^{2}_{loc}(\R_{+}; H^1) \cap L^1_{loc}(\R_{+};B^{1}_{\infty,1})\big)^2,\\[1ex]
\theta \in C_{b}(\R_{+}; B^{0}_{2,1}) ~~\mbox{ and }~~\nabla \Pi \in \big( L^1_{loc}(\R_{+};B^{0}_{2,1})\big)^2.
\end{array}
\end{equation}
Moreover, for any $t>0$, there exists a constant $C_{0}$ depending only on the initial data  such that 
\begin{equation*} 
\|u\|_{L^{1}_{t}(B^{1}_{\infty,1})} \leq C_{0} e^{C_{0} t^4}.
\end{equation*}
\end{thm} 

We claim that the data $(\theta_0,u_0)$ of Theorem \ref{thm:dXBsnq2} fulfill the assumptions of the above theorem. 
Indeed, decompose $u_0 = \Delta_{-1} u_0 + ({\rm Id}-\Delta_{-1})u_0$ and apply the following \emph{Biot-Savart law} 
 \begin{equation*}
 \nabla u=-\nabla (-\Delta)^{-1}\nabla^\perp\omega\quad\hbox{with}\quad
 \nabla^\perp :=(-\d_2,\d_1).
 \end{equation*}
Then thanks to the obvious embedding $B^{\frac{2}{q}-2}_{q,1}(\R^2) \hookrightarrow B^{-1}_{2,1}(\R^2),$  we obtain that
\begin{align}\label{es:u0_Bsnq2}
\|u_0\|_{B^{-1}_{\infty,1}} \lesssim \|u_0\|_{B^{0}_{2,1}} & \lesssim \| \Delta_{-1} u_0\|_{L^2} + \|({\rm Id}-\Delta_{-1})\nabla u_0\|_{B^{-1}_{2,1}} \nonumber\\
                                                           & \lesssim \|u_0\|_{L^2} + \|\omega_0\|_{B^{\frac{2}{q}-2}_{q,1}}.
\end{align}
Besides, the obvious embedding $\theta_0 \in B^{\frac{2}{q}-1}_{q,1}(\R^2) \hookrightarrow B^{0}_{2,1} (\R^2)$ holds. 
Hence one may apply the above theorem to our data. 
The corresponding global solution $(\theta,u)$ fulfills \eqref{spc:AH2007} and 
the following inequality for all $t\geq0,$
\begin{equation} \label{es:u_lip1_Bsnq2}
\|\nabla u\|_{L^{1}_{t}(L^{\infty})} \lesssim \|u\|_{L^{1}_{t}(B^{1}_{\infty,1})} \leq C_{0} e^{C_{0} t^4}.
\end{equation}

%%%%%%%%%%%%%%%%%%%%%%%%%%%%%%%%%%%%%%%

\subsection{A priori estimates for $\theta$ and $\omega$}

We now want to prove that $(\theta,u)$ fulfills the additional property \eqref{spc:thm_dXBsnq2}.
To this end, the first observation is that  $\theta$ satisfies a free transport equation. Hence, from the standard theory of transport equations (apply  Proposition \ref{prop:trans_diff} with $\nu=0$) 
and the bound \eqref{es:u_lip1_Bsnq2}, we deduce that
\begin{equation}\label{es:theta_Bsnq2}
\|\theta\|_{\widetilde{L}^{\infty}_{t}(B^{\frac{2}{q}-1}_{q,1})} \leq e^{C\|\nabla u\|_{L^1_t (L^\infty)}}\|\theta_0\|_{B^{\frac{2}{q}-1}_{q,1}} \leq 
\|\theta_0\|_{B^{\frac{2}{q}-1}_{q,1}}\exp \big( C_0\exp (C_{0} t^4) \big).
\end{equation}
In order to bound $\omega,$ one may apply   Proposition  \ref{prop:trans_diff}
to the vorticity equation, which yields  for all $t\geq0$,
$$
\|\omega\|_{\widetilde{L}^{\infty}_t(B^{\frac{2}{q}-2}_{q,1})} + \|\omega\|_{L^{1}_t(B^{\frac{2}{q}}_{q,1})} \leq C (1 +t)e^{C(1+  t) \|\nabla u\|_{L^1_{t}(L^{\infty})}} \big(\|\omega_0\|_{B^{\frac{2}{q}-2}_{q,1}} + \|\nabla \theta\|_{L^{1}_{t} (B^{\frac{2}{q}-2}_{q,1})} \big)\cdotp
$$
Hence, taking advantage of \eqref{es:u_lip1_Bsnq2} and \eqref{es:theta_Bsnq2}, we get
\begin{equation}\label{es:vorticity_Bsnq2}
\|\omega\|_{L^{\infty}_t(B^{\frac{2}{q}-2}_{q,1})} + \|\omega\|_{L^{1}_t(B^{\frac{2}{q}}_{q,1})} \leq C\exp(\exp(C_0t^4))
 \big(\|\omega_0\|_{B^{\frac{2}{q}-2}_{q,1}} + \|\theta_0\|_{B^{\frac{2}{q}-1}_{q,1}} \big).
\end{equation}
Then   Biot-Savart law allows to improve the regularity of the velocity $u$ as follows:  
\begin{equation}\label{es:u_lip2_Bsnq2}
 U_q(t):=\|\nabla u\|_{L^{1}_{t}(B^{\frac{2}{q}}_{q,1})} \lesssim \|\omega\|_{L^1_{t}(B^{\frac{2}{q}}_{q,1})}  \leq C_{0} \exp \big( \exp (C_{0} t^4) \big),
\end{equation}
where $C_0$ depends only on the  Lebesgue and Besov norms of the  data in Theorem \ref{thm:dXBsnq2}.

%%%%%%%%%%%%%%%%%%%%%%%%%%%%%%%

\subsection{A priori estimates for the striated regularity}

We shall need the  following lemma which is a straightforward generalization of Inequality \eqref{es:dXu_divXomega1}. 
\begin{lemma}
For any $\ep \in ]0,1[$, there exists a constant $C$ such that the following estimate holds true:
\begin{equation} \label{es:dXu_divXomega2}
\|\d_{X}u\|_{\widetilde{L}^{1}_t (\sC^{\ep})} \lesssim \int_{0}^{t}\|\nabla u\|_{L^{\infty}} \|X\|_{\sC^{\ep}}\, dt' + \|\div(X\omega)\|_{\widetilde{L}^{1}_t (\sC^{\ep-1})}.
\end{equation}
\end{lemma}
As we explained in the introduction, the above lemma implies that it suffices to bound  $\div(X\omega)$ in $\widetilde{L}^{1}_t (\sC^{\ep-1})$ for all $t\geq0$ in order to  propagate the H\"{o}lder regularity $\sC^{\ep}$ of $X.$  This will be based on Proposition \ref{prop:trans_diff}, remembering that  $\div(X\omega)$ fulfills the  transport-diffusion equation:  
$$
\d_t\div(X\omega)+u\cdot\nabla\div(X\omega)-\Delta\div(X\omega)=f
$$
with 
$$
f= \div F+\div(X\d_1\theta)\quad\hbox{and}\quad
F := X\Delta \omega - \Delta(X\omega).$$
Thanks to Proposition \ref{prop:trans_diff}, we have 
\begin{multline}\label{es:divXw_Bsnq2}
\|\div (X\omega)\|_{L^{\infty}_{t}(\sC^{\ep-3})}
+\|\div (X\omega)\|_{\widetilde{L}^{1}_{t}(\sC^{\ep-1})} \lesssim  (1+ t)
\Bigl(\|\div (X_0 \omega_0)\|_{\sC^{\ep-3}} \\+ \int_{0}^{t} U'_{q}\, \|\div(X\omega)\|_{\sC^{\ep-3}}\,dt' + \|f\|_{\widetilde{L}^{1}_{t}(\sC^{\ep-3})}\Bigr),
\end{multline}
where $U_q$ has been defined in \eqref{es:u_lip2_Bsnq2}.
\medbreak
Let us now  bound the source term $f$ in $\widetilde{L}^{1}_{t}(\sC^{\ep-3})$ of  \eqref{es:divXw_Bsnq2}. It is not hard to  estimate  $F$ via the following decomposition:
\begin{equation*}
F = [T_X,\Delta] \omega +T_{\Delta \omega} X + R(X,\Delta \omega)-\Delta T_{\omega} X -\Delta R(\omega, X).
\end{equation*}
Using the commutator estimates of Lemma \ref{lemma:ce1} and 
standard results of continuity for the paraproduct and remainder operators (see e.g. \cite{BCD2011}, Chap. 2), we get
if  Condition \eqref{cdt:striated_Bsnq2} is fulfilled,
\begin{equation}\label{es:F1_dBsnq2}
\|F\|_{\sC^{\ep-2}} \lesssim \|\omega\|_{B^{\frac{2}{q}}_{q,1}} \|X\|_{\sC^{\ep}}.
\end{equation}
The last  part of $f$ may be decomposed into
 $$\div (X \d_1 \theta) = \d_1 (\div (X \theta))- \div (\theta \d_1 X ).$$
 By Bony's decomposition \eqref{eq:bony}, $$\theta \d_1 X = T_{\theta} \d_1 X + T_{\d_1 X} \theta + R(\theta, \d_1 X).$$ As Condition \eqref{cdt:striated_Bsnq2} is fulfilled, we thus  have 
\begin{equation} \label{es:F2_dBsnq2}
\|\theta \d_1X \|_{\sC^{\ep-2}} \lesssim \|\theta\|_{B^{\frac{2}{q}-1}_{q,1}}  \|\d_1 X\|_{\sC^{\ep-1}}.
\end{equation}
Next, we see that
$$
\div(X\theta)=\d_X\theta+\theta\,\div X.
$$
The last term $\theta\,\div X$ may be bounded as in \eqref{es:F2_dBsnq2}. 
Therefore, combining with  Inequalities \eqref{es:F1_dBsnq2} and \eqref{es:F2_dBsnq2}, integrating with respect to time, and using also  the obvious embedding $ L^{1}_{t}(\sC^{\ep-3})\hookrightarrow \widetilde{L}^{1}_t(\sC^{\ep-3})$, we end up with 
\begin{equation} \label{es:f_dBsnq2}
 \|f\|_{\widetilde{L}_t^{1}(\sC^{\ep-3})}  \lesssim\int^t_0 (\|\omega\|_{B^{\frac{2}{q}}_{q,1}}  + \|\theta\|_{B^{\frac{2}{q}-1}_{q,1}}) \|X\|_{\sC^{\ep}}\, dt' +  \|\d_X \theta \|_{L^{1}_t (\sC^{\ep-2})}.
\end{equation}
Next, resuming to the transport equation \eqref{eq:X} satisfied by $X,$  combining with the last
item of Proposition \ref{prop:trans_diff},  and using also    \eqref{es:dXu_divXomega2} and  \eqref{es:divXw_Bsnq2}, we get 
\begin{align}\label{es:X_Bsnq2}
\|X\|_{L^{\infty}_{t}(\sC^{\ep})} & \leq \|X_0\|_{\sC^{\ep}} + C\int_{0}^{t} \|\nabla u\|_{L^{\infty}} \|X\|_{\sC^{\ep}}\,dt' + C\|\d_X u\|_{\widetilde{L}^{1}_{t}(\sC^{\ep})} \nonumber \\
& \leq \|X_0\|_{\sC^{\ep}} + C\int_{0}^{t} \|\nabla u\|_{L^{\infty}} \|X\|_{\sC^{\ep}}\,dt' + C\|\div (X\omega)\|_{\widetilde{L}^{1}_{t}(\sC^{\ep-1})} \nonumber \\
 & \leq   \|X_0\|_{\sC^{\ep}}+ C(1 + t)\|\div (X_0 \omega_0)\|_{\sC^{\ep-3}} \nonumber \\
&\hspace{2cm} + C(1+t) \Big( \int_{0}^{t} U'_{q}  \big(\|X\|_{\sC^{\ep}}+ \|\div(X \omega)\|_{\sC^{\ep-3}} \big)\,dt' + \|f\|_{\widetilde{L}^{1}_{t}(\sC^{\ep-3})} \Big).
\end{align}
Set \begin{equation*}
Z(t):= \|X\|_{L^{\infty}_{t}(\sC^{\ep})} + \|\div(X \omega)\|_{L^{\infty}_{t}(\sC^{\ep-3})}
\end{equation*} 
and 
\begin{equation*}
W(t) := U_q(t) + \int^{t}_{0} (\|\omega\|_{B^{\frac{2}{q}}_{q,1}}+ \|\theta\|_{B^{\frac{2}{q}-1}_{q,1}})\, dt'.
\end{equation*} 
Observe that  the bounds \eqref{es:theta_Bsnq2}, \eqref{es:vorticity_Bsnq2} and \eqref{es:u_lip2_Bsnq2} imply that
\begin{equation}\label{es:W_dBsnq2}
W(t) \leq C_0 \exp\bigl(\exp(C_0 t^4)\bigr) \big(\|\omega_0\|_{B^{\frac{2}{q}-2}_{q,1}} + \|\theta_0\|_{B^{\frac{2}{q}-1}_{q,1}} \big).
\end{equation}
Then putting together  \eqref{es:divXw_Bsnq2}, \eqref{es:f_dBsnq2} and \eqref{es:X_Bsnq2} yields
\begin{align*}
Z(t) \lesssim (1 + t) \bigg( Z(0) + \|\d_X \theta\|_{L^{1}_{t} (\sC^{\ep-2})} + \int_{0}^{t} W'(t') Z(t')\,dt' \bigg)\cdotp
\end{align*}
Hence, by virtue of  Gronwall lemma,
\begin{equation}\label{es:Z1_dBsnq2}
Z(t)  \leq C (1 + t) (Z(0) + \|\d_X \theta\|_{L^{1}_{t} (\sC^{\ep-2})} ) e^{C(1+ t)W(t)}.  
\end{equation}
Now, Proposition \ref{prop:trans_diff} and the fact that $\d_X\theta$ satisfies a free transport equation imply that
\begin{equation}\label{es:dXtheta_dBsnq2}
\|\d_X \theta\|_{L^{\infty}_{t}(\sC^{\ep-2})} \leq e^{ C U_q(t) }  \|\d_{X_0} \theta_0\|_{\sC^{\ep-2}}.
\end{equation}
Adding \eqref{es:W_dBsnq2} and \eqref{es:dXtheta_dBsnq2} to \eqref{es:Z1_dBsnq2},  we thus obtain
\begin{equation}\label{es:Z_dBsnq2}
Z(t) \leq C_{0} \exp \Big(\exp \big(\exp ( C_{0}t^4)\big) \Big)\cdot
\end{equation}
Resuming to \eqref{es:divXw_Bsnq2}, one can eventually conclude that 
\begin{equation}\label{es:striated}
\|\div (X\omega)\|_{\widetilde{L}^{1}_{t}(\sC^{\ep-1})} + \|\div (X\omega)\|_{{L}^{\infty}_{t}(\sC^{\ep-3})} 
 \leq C_{0} \exp \Big(\exp \big(\exp ( C_{0}t^4)\big) \Big)\cdot
 \end{equation}

%%%%%%%%%%%%%%%%%%%%%%%%%%%%%%%%

\subsection{Completing the proof of Theorem \ref{thm:dXBsnq2}}\label{sect:complt_thm_dXBsnq2}

In order to make the previous computations rigorous, we have to work on smooth solutions. To this end,  we  solve  System $(B_{1,2})$ with smoothed out initial data 
\begin{equation*}
(\theta^n_0, u^n_0) := (S_n \theta_0, S_n u_0).
\end{equation*}
It is clear by embedding and \eqref{es:u0_Bsnq2} that the components of $(\theta_0,u_0)$ belong to all Sobolev spaces $H^s(\R^2)$. Thanks to the result in \cite{Chae2006}, we thus get  a unique global smooth solution
$(\theta^n,u^n,\nabla\Pi^n)$  having Sobolev regularity  of any order.
Furthermore, as $\theta^n_0$ belongs to all spaces $B^s_{q,1}$
and satisfies a linear transport equation with a smooth velocity field,  we are guaranteed that $\theta^n\in C(\R_+; B^s_{q,1})$ for all $s\in\R.$  
Then as $\omega^n$ is the solution of the transport-diffusion equation with $\omega^n_0\in B^s_{q,1}(\R^2)$ and source term $\d_1\theta^n$ in $C\big(\R_+; B^s_{q,1}(\R^2)\big)$ for all $s\in\R,$ we conclude that $(\theta^n,\omega^n)$ belongs to all sets  $E^s_{q}$ defined by 
$$\displaylines{ E^s_q:= \Bigl\{(\vartheta,\sigma): \vartheta \in {L}^{\infty}_{loc}(\R_{+}; B^{\frac{2}{q}-1 +s}_{q,1}),\  \sigma \in  \widetilde{L}^\infty_{loc}(\R_{+};B^{\frac{2}{q}-2+s}_{q,1}) \cap L^1_{loc}(\R_{+};B^{\frac{2}{q}+s}_{q,1}) \Bigr\}\cdotp}$$
Finally, regularizing $X_0$ into $X_0^n:=S_{n}X_0$ and setting  $ X^n(t,x) := (\d_{X_0^n} \psi_{u^n})(\psi^{-1}_{u^n}(t,x)),$  we see that $X^n$ belongs to H\"older spaces of any order, and satisfies \eqref{eq:X} with velocity field $u^n.$ The estimates that we proved so far are thus valid  for $(\theta^n,u^n, X^n).$
In particular,  Inequalities  \eqref{es:theta_Bsnq2}, \eqref{es:W_dBsnq2} and \eqref{es:Z_dBsnq2} are satisfied for all $n\in\N,$
\emph{with their r.h.s. depending on $n$ through the (regularized) initial data respectively}. 
 \medbreak
Deducing from  \eqref{es:theta_Bsnq2} and  \eqref{es:W_dBsnq2} that  the sequence  $(\theta^n, \omega^n)_{n\in\N}$ is bounded in $E^0_q$ is obvious as $S_n$ maps $L^p$ to itself for all $n\in\N,$ with a norm independent of $n.$
This guarantees that  
\begin{equation}\label{es:unif_dBnsq2}
\|\theta_0^n\|_{B^{\frac{2}{q}-1}_{q,1}} \lesssim \|\theta_0\|_{B^{\frac{2}{q}-1}_{q,1}},\quad 
\|u^n_0\|_{L^2} \lesssim \|u_0\|_{L^2} \quad \hbox{and} \quad
\|\omega^n_0\|_{B^{\frac{2}{q}-2}_{q,1}} \lesssim  \|\omega_0\|_{B^{\frac{2}{q}-2}_{q,1}}.  
\end{equation}
To justify the uniform boundedness of $(X^n, \d_{X^n} \theta^n, \div(X^n\omega^n))$ in
\begin{equation*}
\big(L^\infty_{loc}(\R_+; \sC^\ep) \big)^2 \times L^\infty_{loc} (\R_{+}; \sC^{\ep-2}) \times \big(L^\infty_{loc}(\R_{+}; \sC^{\ep-3}) \cap \widetilde{L}^{1}_{loc}(\R_{+}; \sC^{\ep-1})\big),
\end{equation*}
it is only a matter of checking  that the `constant' $C^n_0$ that appears in the r.h.s. of \eqref{es:dXtheta_dBsnq2} and \eqref{es:Z_dBsnq2} could be bounded independently of $n.$ 
In fact, besides the norms appearing in  \eqref{es:unif_dBnsq2}, 
 $C^n_0$  depends only (continuously) on
\begin{equation*}
\|X_0^n\|_{\sC^{\ep}},\quad
\|\d_{X_0^n} \theta^n_0\|_{\sC^{\ep-2}} \quad \hbox{and} \quad
\|\div(X^n_0\omega_0^n )\|_{\sC^{\ep-3}}.  
\end{equation*}
Arguing as in \eqref{es:unif_dBnsq2}, we see that $\|X_0^n\|_{\sC^{\ep}}$ can be uniformly controlled by $\|X_0\|_{\sC^{\ep}}.$ 
 Furthermore, combining Lemma \ref{lemma:ce3} and Lemma 2.97 in \cite{BCD2011}, we get
$$\|\d_{X_0^n} \theta^n_0 \|_{\sC^{\ep-2}} \lesssim \|\d_{X_0} \theta_0\|_{\sC^{\ep-2}} + \|\theta_0\|_{B^{\frac{2}{q}-1}_{q,1}}\|X_0\|_{\sC^\ep}.$$
Finally, we claim that
$$ \|\div(X^n_0  \omega_0^n )\|_{\sC^{\ep-3}} \lesssim \|\div(X_0 \omega_0 )\|_{\sC^{\ep-3}} 
+ \|\omega_0\|_{B^{\frac{2}{q}-2}_{q,1}}\|X_0\|_{\sC^\ep}.$$
This is a consequence of the  decomposition
$$\div (Yf)-\cT_{Y}f = \div (T_{f} Y + R(f,Y))+[\d_k, T_{Y^k}]f,$$
applied to  $(Y,f) = (X^n_0,\omega^n_0) \mbox{ or } (X_0,\omega_0),$ 
and of the fact that 
$$\cT_{X_0^n}\omega_0^n = S_n \cT_{X_0}\omega_0+[T_{(X_0^n)^k},S_n]\d_k\omega_0 + \cT_{(X^n_0-X_0)} \omega^n_0,$$
where the last term $\cT_{(X^n_0-X_0)} \omega^n_0$ vanishes if $N_0$ in \eqref{eq:bony} is taken larger than $1$. 
\medbreak

Let us now establish that $(\theta^n,u^n,\nabla\Pi^n)$ converges (strongly) 
to some solution $(\theta,u,\nabla\Pi)$ of $(B_{1,2})$ belonging to the space  $F_2^0(\R^2)$
where, for all $s\in\R,$ $p \in [1,\infty]$ and $N\geq 2,$ we set
$$\displaylines{F^s_p(\R^N):=\Bigl\{(\vartheta,v,\nabla P): \vartheta \in \widetilde{L}^{\infty}_{loc}(\R_{+}; B^{\frac{N}{p}-1+s}_{p,1}(\R^N)), \nabla P \in \Big(L^1_{loc}\bigl(\R_{+};B^{\frac{N}{p}-1+s}_{p,1}(\R^N)\bigr)\Big)^N 
\hfill\cr\hfill\hbox{and}\quad  v\in\Big(\widetilde{L}^{\infty}_{loc}\big(\R_{+};B^{\frac{N}{p}-1+s}_{p,1}(\R^N)\big) \cap L^1_{loc}\big(\R_{+};B^{\frac{N}{p}+1+s}_{p,1}(\R^N)\big) \Big)^N \Bigr\}\cdotp}$$

To this end, we shall first prove  that $(\theta^n,u^n,\nabla\Pi^n)_{n\in\N}$ is a Cauchy sequence in the space $F^{-1}_2(\R^2).$ Indeed, if setting 
 $$(\delta \theta^m_n, \delta u^m_n, \delta \Pi^m_n) := (\theta^m-\theta^n, u^m-u^n, \Pi^m-\Pi^n),$$
 then we get from $(B_{1,N})$ that 
\begin{equation}\tag{$B^{m,n}_{1,N}$}\label{eq:BsnqmnN}
\left\{
\begin{array}{l}
\d_t \delta \theta^m_n +\div(u^m \delta\theta^m_n) = -\div(\delta u^m_n\; \theta^n), \\[1ex]
\d_t \delta u^m_n +\div(u^m \otimes \delta u^m_n)- \Delta \delta u^m_n 
+ \nabla \delta\Pi^m_n = \delta\theta^m_n\, e_N - \div(\delta u^m_n \otimes u^n),\\[1ex]
\div \delta u^m_n = 0.
\end{array}
\right.
\end{equation}
In dimension  $N=2,$  we infer from  Bony decomposition \eqref{eq:bony} and continuity results
for the paraproduct and remainder,   that 
\begin{equation*}
\|\div(\delta\! u^m_n  \theta^n)\|_{B^{-1}_{2,1}}  \lesssim \|\delta\! u^m_n \|_{B^{1}_{2,1}} \|\theta^n\|_{B^{0}_{2,1}} ~~\mbox{and  }~~\|\div(\delta\! u^m_n\otimes u^n)\|_{B^{-1}_{2,1}}  \lesssim \|\delta\! u^m_n\|_{B^{-1}_{2,1}} \|u^n\|_{B^{2}_{2,1}}. 
\end{equation*}
Using the above two inequalities and the estimates for the transport and transport-diffusion equations 
stated in Proposition \ref{prop:trans_diff}, we end up with 
$$
\begin{aligned}
&\| \delta \theta^m_n \|_{\widetilde{L}^\infty_t (B^{-1}_{2,1})} \leq  e^{C\|\nabla u^{m}\|_{L^1_t(B^{1}_{2,1})}}
  \bigg( \|\delta \theta^m_n(0) \|_{B^{-1}_{2,1}} + \int^t_0 \|\delta u^m_n \|_{B^{1}_{2,1}} \|\theta^n\|_{B^{0}_{2,1}} dt'\bigg),\\
&\sE^m_n(t) \lesssim (1+t^2) e^{C(1+t)\|\nabla u^{m}\|_{L^1_t(B^{1}_{2,1})}} \bigg( \|\delta u^m_n(0) \|_{B^{-1}_{2,1}} +  \|\delta \theta^m_n(0)\|_{B^{-1}_{2,1}} \\ 
&\hspace{7cm}+ \int^t_0 \big( \|u^n\|_{B^{1}_{2,1}}  + \|\theta^n\|_{B^{0}_{2,1}}  \big) \sE^m_n(t')\,dt' \bigg),
\end{aligned}
$$
where $\sE^m_n(t):= \| \delta u^m_n\|_{\widetilde{L}^{\infty}_t(B^{-1}_{2,1})}+\| \delta u^m_n\|_{L^{1}_t(B^{1}_{2,1})}.$  
\medbreak

Combining Gronwall lemma with the uniform bounds that we proved for $(\theta^n,u^n),$ we deduce that 
 $(\theta^n,u^n,\nabla\Pi^n)_{n\in\N}$  strongly converges in the norm of $\|\cdot\|_{F^{-1}_2(\R^2)}.$
 Then interpolating with the uniform bounds, we see that strong convergence holds true in  $F^{-\eta}_2(\R^2),$ too,
 for all $\eta>0,$ which allows to justify that $(\theta,u,\nabla\Pi)$ satisfies $(B_{1,2}).$ Taking advantage of  Fatou property of Besov spaces, we gather that  the limit  $ (\theta,u,\nabla\Pi)$ belongs to 
 $F^{0}_2(\R^2).$    In addition, as the sequence $(\theta^n, \omega^n)$ is  bounded in $E^0_q$, 
 Fatou property also implies that  $(\theta, \omega)$ belongs to $E^0_q.$ Finally,  using (complex) interpolation, we obtain that for any $0<\eta,\delta<1$, sequence $(\theta^n, \omega^n)_{n\in\N}$ converges to $(\theta, \omega)$ in the space
$E^{-\delta \eta}_{q_\delta}$ with ${q_\delta}:= \frac{\delta}{2} + \frac{1-\delta}{q}.$
\medbreak

To complete the proof of Theorem \ref{thm:dXBsnq2}, we have to check that the announced proprieties 
of striated regularity are fulfilled. 
In fact, taking advantage of the (uniform) Lipschitz-continuity of  $u^n$, we may obtain that for all  $\eta >0$ (see \cite{Che1998}),
\begin{equation}\label{cv:X_dXBsnq2}
X = \lim_{n \rightarrow \infty} X^n \mbox{  in  } L^{\infty}_{loc}(\R_+;\sC^{\ep-\eta}),
\end{equation}  
which allows to justify  that $X$ satisfies \eqref{eq:X}, and also 
that for all ${\eta}'>0,$  Sequence $\big(\d_{X^n} \theta^n, \div (X^n \omega^n) \big)_{n\in\N}$ 
converges in the space  $$L^{\infty}_{loc}(\R_{+}; \sC^{\ep-2-\eta'}) \times  {L}^{1}_{loc}(\R_{+}; \sC^{\ep-1-\eta'}).$$
As we know in addition that  \eqref{es:dXtheta_dBsnq2} and \eqref{es:striated} are  fulfilled by $\d_{X^n}\theta^n$ and
$\div(X^n\omega^n)$ for all $n\in\N,$ Fatou property 
enables us to conclude that we have 
$$\d_{X} \theta\in L^\infty_{loc}(\R_+;\sC^{\ep-2})\quad\hbox{and}\quad
\div(X\omega)\in L^{\infty}_{loc}(\R_{+}; \sC^{\ep-3}) \times  \wt{L}^{1}_{loc}(\R_{+}; \sC^{\ep-1}),
$$
and that both \eqref{es:dXtheta_dBsnq2} and \eqref{es:striated} hold true.

%%%%%%%%%%%%%%%%%%%%%%%%%%%%%%%%%%%%%%%%%%%%%%%%%%%%%%%%%%%%%%%%

\section{Propagation of striated regularity in  the general case $N \geq 3$}\label{sect:thm_dXBsnqN}

This section is dedicated to the proof of Theorem \ref{thm:dXBsnqN}. The first (easy) step is to  extend the global existence result of R. Danchin and M. Paicu in  \cite{DanP2008a}, to nonhomogeneous Besov spaces. 
The second  step is to propagate striated regularity. As pointed out in the introduction, it suffices to bound $\d_Xu$ in the space $\widetilde{L}^1_t(\sC^\ep),$ which requires our using the smoothing properties of the heat flow.  
By applying the directional derivative $\d_X$ to the velocity equation, we discover
that $\d_Xu$ may be seen as the solution to some evolutionary Stokes system with a \emph{nonhomogeneous}
divergence condition, namely (using the fact that  $\div u=0$), 
$$ \div (\d_X u)= \div \big( \d_u X -u\, \div X \big).$$ 
In order to reduce our study to that of a solution to the standard Stokes system, it is thus natural 
to decompose $\d_Xu$ into
$$\d_Xu=z+\d_uX-u\,\div X.$$
Now, $z$ is indeed divergence free. Unfortunately, the above decomposition 
introduces additional source terms in the Stokes system satisfied by $z.$ One of those terms is  $\Delta\d_u X,$ that
we do not know how to handle without losing regularity. Moreover, this also leads to the failure of global estimates technically.

 To overcome this difficulty,  we here propose to replace the directional derivative $\d_X$ by the \emph{para-vector field} $\cT_X$ that is defined by 
   \begin{equation*}
 \cT_X u:=  T_{X^k} \d_k u.  
 \end{equation*}
 The reason why this change of viewpoint is appropriate is that, as we shall see below,  the aforementioned loss
 of regularity does not occur anymore when estimating  $\cT_X u,$ and that  
 $\cT_X u$ may be seen of  the leading order part  of $\d_X u,$ 
 as shown by the following  two inequalities (that hold true whenever $\frac Np+\ep-1 \geq 0$, see  Lemma \ref{lemma:ce3}): 
\begin{align}\label{es:dXu_TXu_BsnqN1}
\|\d_Xu-\cT_Xu\|_{L^\infty_t(\sC^{\ep-2})} &\lesssim \|X\|_{L^\infty_t(\wt\sC^\ep)}\|u\|_{L^\infty_t( B^{\frac Np-1}_{p,1})},\\\label{es:dXu_TXu_BsnqN2}
\|\d_Xu-\cT_Xu\|_{L^1_t(\sC^{\ep})} &\lesssim \|X\|_{L^\infty_t(\sC^\ep)}\|u\|_{L^1_t(B^{\frac Np+1}_{p,1})}.
\end{align}
Thus the second step of the proof of Theorem \ref{thm:dXBsnqN} will be mainly devoted
to  proving a priori  estimates for $\cT_Xu$ in $\big( L^\infty_t(\sC^{\ep-2})\cap \wt L^1_t(\sC^{\ep})\big)^N.$
\medbreak
Finally, in the last step of the proof, we smooth out the data (as in the 2-D case) so as
to justify that the a priori estimates of the first two steps are indeed satisfied.

\subsection{Global existence  in nonhomogeneous Besov spaces}
The following result is an obvious modification of the work by the first author and M. Paicu in  \cite{DanP2008a}. 
\begin{thm}\label{thm:wellposed_BsnqN}
Let  $N \geq 3$ and $p \in ]N, \infty[$. Assume that $\theta_0 \in B^{0}_{N,1}(\R^N) \cap L^{\frac{N}{3}}(\R^N)$ and that the initial divergence-free velocity  field $u_0$ is in $(B^{\frac{N}{p}-1}_{p,1}(\R^N) \cap L^{N,\infty}(\R^N))^N$. There exists a (small) positive constant $c$ independent of $p$ such that if
\begin{equation}\label{cdt:small_BsnqN}
\|u_0\|_{L^{N,\infty}} + \nu^{-1} \|\theta_0\|_{L^{\frac{N}{3}}} \leq c\nu,
\end{equation}
then the Boussinesq system \eqref{eq:BsnqN} has a unique global solution 
\begin{equation*}
(\theta, u, \nabla \Pi) \in C(\R_{+};B^{0}_{N,1}) \cap \big( C(\R_{+};B^{\frac{N}{p}-1}_{p,1}) \cap L^1_{loc}(\R_{+};B^{\frac{N}{p}+1}_{p,1})\big)^N \cap \big( L^1_{loc}(\R_{+};B^{\frac{N}{p}-1}_{p,1}) \big)^N.
\end{equation*}
Moreover, there exists some positive constant $C$ independent on $N$ such that for any $t\geq0$, we have 
\begin{equation}\label{es:lorentz_BsnqN}
\|\theta(t)\|_{L^{\frac{N}{3}}} = \|\theta_0\|_{L^{\frac{N}{3}}},\quad \|u(t)\|_{L^{N,\infty}} \leq C ( \|u_0\|_{L^{N,\infty}} + \nu^{-1} \|\theta_0\|_{L^{\frac{N}{3}}} ),
\end{equation}
\begin{equation}\label{es:Up_BsnqN}
U_{p}(t) \leq A(t), % \leq C_{0,\nu} e^{C_{0,\nu}t},
\end{equation}

\begin{equation}\label{es:theta_BsnqN}
\|\theta\|_{\widetilde{L}^{\infty}_t (B^{0}_{N,1})} \leq C \|\theta_0\|_{B^{0}_{N,1}}(1+\nu^{-1}A(t)),%\leq C_{0,\nu} e^{C_{0,\nu}t},
\end{equation}

\begin{equation}\label{es:uPi_BsnqN}
\|(\d_t u, \nabla \Pi)\|_{L^{1}_t (B^{\frac{N}{p}-1}_{p,1})} \leq C \nu^{-1} A^2(t)+ Ct\|\theta_0\|_{B^{0}_{N,1}}\big( 1+A(t) \big),
% \leq  C_{0,\nu} e^{C_{0,\nu}t},
\end{equation}
where  

\begin{align*}
U_p(t) & := \|u\|_{\widetilde{L}^{\infty}_t (B^{\frac{N}{p}-1}_{p,1})} + \nu \|u\|_{L^{1}_t (B^{\frac{N}{p}+1}_{p,1})}\quad\hbox{and}\\
A(t) & := C (\|u_0\|_{B^{\frac{N}{p}-1}_{p,1}}\!+\!\nu) e^{C\nu^{-1}t\|\theta_0\|_{B^{0}_{N,1}}} + C\nu^2\biggl( \frac{\|\theta_0\|_{B^{0}_{N,1}}+\nu^2}{\|\theta_0\|_{B^{0}_{N,1}}^2}\biggr)\biggl(e^{C\nu^{-1}t\|\theta_0\|_{B^{0}_{N,1}}}-1\biggr)\cdotp
\end{align*}
%and $C_{0,\nu}$ is some positive constant depending only on initial data and $\nu$.\medbreak
 %In the above statement, the norm $\|\cdot\|_{L^{\frac{N}{3}}}$ can be replaced by $\|\cdot\|_{L^{\frac{N}{3} , \infty}}$ if $N \geq 4$.
\end{thm}

The above nonhomogeneous estimates are  based on a very simple observation allowing to control the low frequency part of the velocity field :  for $r$ satisfying $1+\frac{1}{p}=\frac{1}{N} + \frac{1}{r}$, we have
as a consequence of refined Young's inequalities (see \cite{BCD2011}):
\begin{equation}\label{obs:u_lf_BsnqN}
\|\Delta_{-1} u(t)\|_{L^{p}} \lesssim \|\mathcal{F}^{-1}\chi\|_{L^{r}} \|u(t)\|_{L^{N,\infty}}, 
\end{equation}
whose r.h.s. is bounded according to \eqref{es:lorentz_BsnqN} thanks to Theorem 1.4 in  \cite{DanP2008a}. 
\medbreak
 For the reader  convenience,  let us outline the proof of  \eqref{es:theta_BsnqN} and \eqref{es:uPi_BsnqN}
 for a  smooth solution of \eqref{eq:BsnqN}.
 Let   $\mathbb{P} := {\rm Id}-\nabla \Delta^{-1}\div$ be the Leray projector
onto divergence-free vector fields. 
On  one hand, applying $\mathbb{P} \Delta_j$
 to the momentum equation in \eqref{eq:BsnqN} and noting that Lemma 4.4 in \cite{DanP2008a} still holds for nonhomogeneous Besov space, we get some constant $C$ and sequence $(c_j)_{j\in\N}$ 
 where $\|(c_j)\|_{\ell^1} \leq 1$ such that for any $j \geq 0$ and $t\geq0,$ 
\begin{align}\label{es:u_hf_BsnqN}
\nonumber \|\Delta_j u(t)\|_{L^p} \leq e^{-C \nu t 2^{2j}} \|\Delta_j u_0\|_{L^p}  + & 
\int^{t}_{0} e^{-C \nu (t-t') 2^{2j}} \big( \|\Delta_j \theta \|_{L^p}  \\
& +C2^{j(\frac{N}{p}-1)}c_j \|u\|_{\sC^{-1}} \|u\|_{B^{\frac{N}{p}+1}_{p,1}} \big) \,dt'.
\end{align}
On the other hand, the low frequency part of $u$ could be controlled according to \eqref{obs:u_lf_BsnqN}, 
and we get assuming  the smallness condition \eqref{cdt:small_BsnqN} on initial data,
\begin{eqnarray}\label{es:u_lf_BsnqN}
\|\Delta_{-1}u\|_{L^{\infty}_t (L^p)} + \nu \|\Delta_{-1}u\|_{L^{1}_t (L^p)} &\!\!\!\lesssim\!\!\! &(1+\nu t)\|u\|_{L^{\infty}_t (L^{N,\infty})}\nonumber\\
&\!\!\!\lesssim\!\!\!& (1+\nu t)\bigl( \|u_0\|_{L^{N,\infty}} + \nu^{-1} \|\theta_0\|_{L^{\frac{N}{3}}}\bigr).
\end{eqnarray}
Putting \eqref{es:u_hf_BsnqN} and \eqref{es:u_lf_BsnqN} together thus implies  that 
\begin{equation*}
U_{p}(t) \lesssim \|u_0\|_{B^{\frac{N}{p}-1}_{p,1}} + \|\theta\|_{L^1_t (B^{\frac{N}{p}-1}_{p,1})}
+\nu^{-1}\|u\|_{L^\infty_t(\sC^{-1})}U_p(t) +(1+\nu t)\bigl( \|u_0\|_{L^{N,\infty}} + \nu^{-1} \|\theta_0\|_{L^{\frac{N}{3}}}\bigr).
\end{equation*}
As $L^{N,\infty}\hookrightarrow\sC^{-1}$ and \eqref{es:lorentz_BsnqN} is
fulfilled,  the third  term of the r.h.s. above may be absorbed by the l.h.s. if $c$ is small enough in  \eqref{cdt:small_BsnqN}.
\medbreak
As regards $\theta,$ thanks to Proposition \ref{prop:trans_vishik} and to the 
embedding $B^0_{N,1}\hookrightarrow B^{\frac Np-1}_{p,1}$ for $p\geq N,$  we have for all $t\geq0,$
\begin{equation}\label{es:vishik_BsnqN}
\|\theta(t)\|_{B^{\frac{N}{p}-1}_{p,1}}\lesssim \|\theta\|_{\widetilde{L}^{\infty}_t (B^{0}_{N,1})} \lesssim \|\theta_0\|_{B^{0}_{N,1}} (1+\|\nabla u\|_{L^1_t (L^{\infty})}).
\end{equation}
Adding up the  above estimates to the bound of $U_p (t)$ yields
\begin{equation*}
U_p(t) \leq C (\|u_0\|_{B^{\frac{N}{p}-1}_{p,1}}+c\nu) + C (\|\theta_0\|_{B^{0}_{N,1}}+c\nu^2)t + C\nu^{-1}\|\theta_0\|_{B^{0}_{N,1}} \int^{t}_{0} U_p\,dt'.
\end{equation*}
The fact that $\int^{t}_{0} t'e^{B(t-t')}\,dt' \leq \frac{1}{B^2} (e^{Bt}-1) $ for any positive constant $B$, allows to bound the term $U_p(t)$ by $A(t)$ according to  Gronwall lemma.
\smallbreak
The estimates for  $\nabla \Pi$ can be  deduced from the equation, as
\begin{equation*}
\nabla \Pi = - \mathbb{Q} (u\cdot \nabla u + \theta e_N)\quad\hbox{with}\quad 
\mathbb{Q} := {\rm Id} - \mathbb{P}.  
\end{equation*}
Finally, the bound for $\d_t u$  may be determined  from the momentum equation.

%%%%%%%%%%%%%%%%%%%%%%%%%%%%%%%%%%%%%

\subsection{A priori estimates for striated regularity}

In this subsection, we assume that $\nu=1$ for simplicity (which is not restrictive, owing to \eqref{change:Bsnq}). 

As explained at the beginning of this section, we shall  concentrate on the proof of 
estimates for $\cT_X u$ in  $L^\infty_t(\sC^{\ep-2})\cap \wt L^1_t(\sC^{\ep}),$ for all $t\geq0,$ and this will be (mainly) based on the smoothing properties of the heat flow.  More precisely,  applying the para-vector field  $\cT_X$ to the velocity equation of $(B_{1,N}),$ we discover that
\begin{equation*}
\d_t \cT_X u +u\cdot \nabla \cT_X u-\Delta \cT_X u + \nabla \cT_X \Pi = g,
\end{equation*}
with
  \begin{equation}\label{eq:g_PdXBsnqN}
  g:= - [\cT_{X}, \d_t+u\cdot\nabla]u +[\cT_X,\Delta] u - [\cT_X, \nabla] \Pi + (\cT_{X} \theta )e_N.
  \end{equation}
In general, the divergence-free property is not satisfied by $\cT_Xu$ as 
$$\div \cT_X u = \div( T_{\d_k X} u^k-T_{\div X}u).$$
In order to enter into the standard theory for the Stokes system, we  set 
 $$ v:= \cT_X u - w\quad\hbox{with}\quad w:= T_{\d_k X} u^k-T_{\div X}u.$$
Now,  $v$ satisfies:
\begin{equation} \label{eq:v_PdXBsnqN} \tag{$S$}
\left\{ \begin{array}{l}
\d_t v - \Delta v + \nabla \cT_X \Pi = \widetilde{g},\\[1ex]
\div v = 0,\\
v|_{t=0}=v_0,
\end{array}\right.
\end{equation}
with $\widetilde{g} := g- u\cdot \nabla \cT_X u -(\d_t w - \Delta w)$ and $g$ defined in \eqref{eq:g_PdXBsnqN}. 
\medbreak
Based on that observation, proving a priori estimates for striated regularity involves three steps.
The first step is dedicated to bounding $\widetilde{g}$ (which mainly requires the commutator estimates
of the appendix). In the second step, we take advantage of the smoothing effect of the heat
flow so as to estimate $v.$ 
In the third step, we resume to $\cT_X u$ and also bound $X$ 
so as to complete the proof of our study of striated regularity. 

\subsubsection*{First step: bounds of $\wt g$}

To start with, let us  bound  $g$ in $\big( \widetilde{L}^1_t(\sC^{\ep-2})\big)^N$ for any positive $t.$  According to Proposition \ref{prop:ce_dBnsqN} and to the remark that follows, the  first commutator of $g$ satisfies
\begin{multline}\label{es:g1_PdXBsnqN}
\|[\cT_X, \d_t+u\cdot \nabla] u\|_{\widetilde{L}^1_t(\sC^{\ep-2})}  \lesssim  \|u\|_{L^{\infty}_t (\sC^{-1})}\|\cT_X u\|_{\widetilde{L}^1_t (\sC^{\ep})} \\+\int^t_0  \|u\|_{B^{\frac{N}{p}+1}_{p,1}} \|\cT_{X}u\|_{\sC^{\ep-2}}\,dt' 
  + \int_0^t  \|X\|_{\widetilde{\sC}^{\ep}} \|u\|_{B^{\frac{N}{p}+1}_{p,1}} \|u\|_{B^{\frac{N}{p}-1}_{p,1}}\,dt'.
\end{multline}

Next, thanks to the commutator estimates \eqref{es:ce2} we have the following two bounds,

\begin{equation}\label{es:g2_PdXBsnqN}
 \|[\cT_X,\Delta] u\|_{L^1_t (\sC^{\ep-2})} \lesssim  \int^t_0  \|\nabla X\|_{\sC^{\ep-1}}\|\nabla u\|_{\sC^0}\,dt',
\end{equation}
\begin{equation}\label{es:g3_PdXBsnqN}
 \|[\cT_X, \nabla ]\Pi\|_{L^1_t (\sC^{\ep-2})} \lesssim \int_0^t  \|\nabla X\|_{\sC^{\ep-1}} \|\nabla \Pi\|_{\sC^{-1}}\,dt'.
\end{equation}

Taking  Lemma \ref{lemma:ce3} into account and integrating from $0$ to $t$ gives
\begin{align}\label{es:g4_PdXBsnqN}
 \|\cT_X \theta\|_{L^1_t (\sC^{\ep-2})} & \lesssim  \|\d_X \theta\|_{L^1_t (\sC^{\ep-2})}
  + \int^t_0 \|X\|_{\widetilde{\sC}^{\ep}} \|\theta\|_{B^{\frac{N}{p}-1}_{p,1}}\,dt'.
\end{align}

Putting together \eqref{es:g1_PdXBsnqN} to \eqref{es:g4_PdXBsnqN} thus yields

\begin{multline*}
\|g\|_{\widetilde{L}^1_t(\sC^{\ep-2})}  \lesssim  \|u\|_{L^{\infty}_t (\sC^{-1})}\|\cT_X u\|_{\widetilde{L}^1_t (\sC^{\ep})}+\int^t_0  \|u\|_{B^{\frac{N}{p}+1}_{p,1}} \|\cT_{X}u\|_{\sC^{\ep-2}}\,dt'  \\
  + \int_0^t  \big( \|u\|_{B^{\frac{N}{p}+1}_{p,1}} (\|u\|_{B^{\frac{N}{p}-1}_{p,1}}+1) + \|\nabla \Pi\|_{\sC^{-1}} + \|\theta\|_{B^{\frac{N}{p}-1}_{p,1}} \big) \|X\|_{\widetilde{\sC}^{\ep}} \,dt'+\|\d_X \theta\|_{L^1_t (\sC^{\ep-2})} .
\end{multline*}

Bounding the second term  of $\widetilde{g}$ is obvious : taking advantage of  Bony's decomposition,
and remembering that $\frac{N}{p}+\ep > 1 $ and that  $\div u =0$, we get
\begin{align*}
\|u\cdot \nabla \cT_X u\|_{\widetilde{L}^1_t (\sC^{\ep-2})} & \lesssim \|u\|_{{L}^{\infty}_t(\sC^{-1})}\|\cT_{X}u\|_{\widetilde{L}^1_t (\sC^\ep)} +\int^t_0\|u\|_{B^{\frac{N}{p}+1}_{p,1}}\|\cT_{X}u\|_{\sC^{\ep-2}}\,dt'.
\end{align*} 
Finally,  to bound the term $\d_t w - \Delta w $ in $\widetilde{L}^1_t(\sC^{\ep-2}),$ we use the decomposition
$$
\d_tw-\Delta w= \sum_{\alpha=1}^3 W_{\alpha},
$$
with 
$$
W_1:=T_{\d_k X}\d_tu^k-T_{\div X}\d_tu,\quad
W_2:=T_{\d_k \d_tX}u^k-T_{\div\d_tX}u, \quad
W_3:=\Delta\bigl(T_{\div X}u-T_{\d_k X}u^k\bigr).
$$
By embedding, it suffices to bound the components of $W_\alpha$ in $L^1_t(B^{\frac{N}{p}+\ep-2}_{p,1})$  with $\alpha=1,2,3.$ 
\medbreak
Now, standard continuity results for the paraproduct ensure that
$$
\begin{aligned}
\|W_1 \|_{L^1_t (B^{\frac{N}{p}+\ep-2}_{p,1})} & \lesssim \int^t_0 \|\nabla  X \|_{\sC^{\ep-1}} \|\d_t u\|_{B^{\frac{N}{p}-1}_{p,1}}\,dt',\\
\|W_2 \|_{L^1_t (B^{\frac{N}{p}+\ep-2}_{p,1})}  &\lesssim \int^t_0 \|\d_t X\|_{\sC^{\ep-2}}\|u\|_{B^{{\frac Np}+1}_{p,1}}\,dt',\\
\| W_3\|_{L^1_t(B^{\frac{N}{p}+\ep-2}_{p,1})} &\lesssim\int^t_0  \| \nabla X \|_{\sC^{\ep-1}}\|u\|_{B^{\frac{N}{p}+1}_{p,1}}\,dt'.
\end{aligned}$$
To estimate $\d_tX$ in the bound of $W_2$, we use the fact that
$$
\d_tX=-u\cdot\nabla X+\d_Xu.
$$
The  convection term $u \cdot \nabla X=\div(u\otimes X)$ may be easily bounded (use Bony's decomposition 
and continuity results for the remainder and paraproduct operators) as follows:
\begin{equation*}
\|u \cdot \nabla X\|_{\sC^{\ep-2}} \lesssim \|u\|_{B^{\frac{N}{p}-1}_{p,1}} \|X\|_{\sC^{\ep}}.
\end{equation*}
Bounding $\d_Xu$ according to  Lemma \ref{lemma:ce3} yields 
\begin{equation*}
\|W_2 \|_{L^1_t (B^{\frac{N}{p}+\ep-2}_{p,1})}   \lesssim \int^t_0 (\|X\|_{\widetilde{\sC}^{\ep}}\|u\|_{B^{\frac{N}{p}-1}_{p,1}} + \|\cT_X u\|_{\sC^{\ep-2}}) \|u\|_{B^{\frac{N}{p}+1}_{p,1}}\, dt'.
\end{equation*}
 Putting together all the above estimates  and using (slightly abusively) the notation $U'_p(\tau):= \|u(\tau)\|_{B^{\frac{N}{p} +1}_{p,1}}$ for any $\tau\geq0,$  we eventually get the following estimate for $\wt g$:
\begin{multline}\label{es:tg_PdXBsnqN}
\|\widetilde{g}\|_{\widetilde{L}^1_t(\sC^{\ep-2})} \lesssim  \|u\|_{L^{\infty}_t (\sC^{-1})}\|\cT_X u\|_{\widetilde{L}^1_t (\sC^{\ep})}\!+\!\|\d_X \theta\|_{L^1_t (\sC^{\ep-2})}\! +\!  \int^t_0 U'_p \|\cT_{X}u\|_{\sC^{\ep-2}}\,dt'  \hfill\cr\hfill
  + \int_0^t  \big( U'_p(\|u\|_{B^{\frac{N}{p}-1}_{p,1}}+1) + \|(\d_t u,\nabla\Pi, \theta)\|_{B^{\frac{N}{p}-1}_{p,1}} \big) \|X\|_{\widetilde{\sC}^{\ep}}\, dt'.
  \end{multline}

\subsubsection*{Second step: Bounds for $v$}

The second step is devoted to bounding $v$ in  $\big( L^\infty_t(\sC^{\ep-2})\cap \wt L^1_t(\sC^{\ep}) \big)^N.$ 
  Granted with \eqref{es:tg_PdXBsnqN}, this will essentially follow from the smoothing properties of the heat flow
after spectral localization. More precisely,  
applying  $\mathbb{P}\Delta_j$ to the equation \eqref{eq:v_PdXBsnqN} yields for all $j\geq-1,$
\begin{equation*}
\left\{ \begin{array}{l}
\d_t \Delta_jv - \Delta \Delta_jv = \mathbb{P}\Delta_j\widetilde{g}\\[1ex]
\Delta_jv|_{t=0} = \Delta_jv_{0}.
\end{array}\right.
\end{equation*}
Now, Lemma 2.1 in  \cite{Che1999} implies that if $j\geq0,$
\begin{align*}
\|\Delta_jv(t)\|_{L^{\infty}} \leq e^{-ct 2^{2j}} \|\Delta_jv_{0}\|_{L^{\infty}} + \int^{t}_{0} e^{-c(t-t') 2^{2j}} 
\|\Delta_j\widetilde{g}(t')\|_{L^{\infty}}\,dt'.
\end{align*}
Therefore,  taking the supremum over $j\geq0,$ we find that the high frequency part of $v$ satisfies
\begin{equation*}
\sup_{j\geq 0}2^{j(\ep-2)}\|\Delta_jv\|_{L^{\infty}_t (L^{\infty})}+ \sup_{j\geq 0}2^{j\ep}\|\Delta_j v\|_{L^{1}_t (L^{\infty})} \lesssim \|v_0\|_{\sC^{\ep-2}} + \|\widetilde{g}\|_{\widetilde{L}^1_t(\sC^{\ep-2})}.
\end{equation*}
Bounding  the norm $\| v_0\|_{\sC^{\ep-2}}$ stems from the definition of $v_0$ and from Lemma \ref{lemma:ce3}: we get
\begin{equation}\label{es:v0_PXdBsnqN}
\|v_0\|_{\sC^{\ep-2}} \lesssim \|\cT_{X_0} u_0\|_{\sC^{\ep-2}} + \| u_0\|_{\sC^{-1}}\|X_0\|_{\sC^{\ep}} \lesssim  \|\d_{X_0} u_0\|_{\sC^{\ep-2}} + \| u_0\|_{B^{\frac{N}{p}-1}_{p,1}}\|X_0\|_{\widetilde{\sC}^{\ep}}.
\end{equation}

To handle the low frequencies of $v,$ we first observe from the definition of $\|\cdot\|_{\sC^s}$ that
$$
\|\Delta_{-1}v\|_{ L^{\infty}_t (L^\infty)} +  \|\Delta_{-1} v\|_{L^1_t(L^{\infty})} 
\lesssim \|v\|_{ L^{\infty}_t (\sC^{-2})} +  \|v\|_{L^1_t(L^{\infty})}.
$$
Then we recall the definition $v = \cT_X u - T_{\d_k X} u^k + T_{\div X}u,$ and hence continuity results for paraproduct yield 
\begin{equation*}
\|v\|_{ L^{\infty}_t (\sC^{-2})} +  \|v\|_{L^1_t(L^{\infty})} \lesssim \|X\|_{L^{\infty}_t (L^\infty)} (\|u\|_{\widetilde{L}^{\infty}_t(B^{\frac{N}{p}-1}_{p,1})}+\|u\|_{L^{1}_t(B^{\frac{N}{p}+1}_{p,1})}).
\end{equation*}  
Moreover, the fact that $X$ satisfies the transport equation \eqref{eq:X} ensures
\begin{equation*}
\|X\|_{L^{\infty}_t (L^\infty)} \leq \|X_0\|_{L^{\infty}} e^{\|\nabla u\|_{L^{1}_t(L^\infty)}}.
\end{equation*} 
Hence, using the notation $U_p$ of Theorem \ref{thm:wellposed_BsnqN}, we control the low frequency part of $v$ via
\begin{equation*}
\|\Delta_{-1}v\|_{ L^{\infty}_t (L^\infty)} +  \|\Delta_{-1} v\|_{L^1_t(L^{\infty})} \lesssim \|X_0\|_{L^{\infty}}
U_p(t) e^{CU_p(t)}.
\end{equation*}
Putting the estimates for low and high frequency parts of $v$ together with \eqref{es:v0_PXdBsnqN}, we end up with
\begin{equation}\label{es:tsH_PdXBsnqND}
\widetilde{\sH}(t)\lesssim \|\d_{X_0} u_0\|_{\sC^{\ep-2}} + \| u_0\|_{B^{\frac{N}{p}-1}_{p,1}}\|X_0\|_{\widetilde{\sC}^{\ep}} + \|\widetilde{g}\|_{\widetilde{L}^1_t(\sC^{\ep-2})}+
 \|X_0\|_{L^{\infty}}
U_p(t) e^{CU_p(t)},
\end{equation}
where we denoted
\begin{equation*}
\widetilde{\sH}(t):= \|v\|_{L^{\infty}_t (\sC^{\ep-2})}+\|v\|_{\widetilde {L}^{1}_t(\sC^{\ep})}.
\end{equation*}

\subsubsection*{Third step: bounds for striated regularity}
Keeping \eqref{es:dXu_TXu_BsnqN1} and  \eqref{es:dXu_TXu_BsnqN2} in mind, the core of the proof
consists in  deriving appropriate  bounds for  $\cT_X u$ and $X.$ Now, remembering that
$$\cT_X u=v+w\quad\hbox{with}\quad w=T_{\d_k X}u^k - T_{\div X}u,$$
it is easy to bound the following quantity:
\begin{equation*}
\sH(t) := \|\cT_X u\|_{L^{\infty}_t (\sC^{\ep-2})}+\|\cT_X u\|_{\widetilde {L}^{1}_t(\sC^{\ep})}.
\end{equation*}
Indeed, continuity results for paraproduct operators guarantee that
\begin{equation}\label{es:w}
\|w\|_{L^{\infty}_t (\sC^{\ep-2})}+\|w\|_{\widetilde {L}^{1}_t(\sC^{\ep})} \lesssim  \|u\|_{L^{\infty}_t (\sC^{-1})}\|X\|_{L^{\infty}_t (\sC^{\ep})} + \int^t_0 \|u\|_{\sC^1} \|\nabla X\|_{\sC^{\ep-1}}\,dt'.
\end{equation}
Let us now bound $X$ in $L^\infty_t (\widetilde{\sC}^{\ep})$. As $X$ and $\div X$ satisfy \eqref{eq:X} and  \eqref{eq:divX},
respectively, H\"older estimates for transport equations 
 and Lemma \ref{lemma:ce3} imply that 
\begin{align}\label{es:X_PdXBsnqN}
\| X \|_{L^{\infty}_t (\widetilde{\sC}^{\ep})} & \leq \|X_0\|_{\widetilde{\sC}^{\ep}} + \int^t_0 \|\nabla u\|_{L^\infty} \| X \|_{\widetilde{\sC}^{\ep}}\, dt' + \|\d_X u\|_{\widetilde{L}^1_t (\sC^{\ep})} \nonumber\\
& \leq \|X_0\|_{\widetilde{\sC}^{\ep}} + \int^t_0 U_p' \| X \|_{\widetilde{\sC}^{\ep}}\, dt' + \|\cT_X u\|_{\widetilde{L}^1_t (\sC^{\ep})}.
\end{align}
Resuming to \eqref{es:w}, one may refine the bound of $w$ as follows:
\begin{equation}\label{es:w_PdXBsnqN}
\|w\|_{L^{\infty}_t (\sC^{\ep-2})}+\|w\|_{\widetilde {L}^{1}_t(\sC^{\ep})} \lesssim  \|u\|_{L^{\infty}_t (\sC^{-1})}\|\cT_X u\|_{\widetilde{L}^{1}_t (\sC^{\ep})} + \|X_0\|_{\widetilde{\sC}^\ep} + \int^t_0 U_p'\| X\|_{\widetilde{\sC}^{\ep}}\,dt'.
\end{equation}

By the embedding $L^{N,\infty}(\R^N) \hookrightarrow \sC^{-1}(\R^N)$ and smallness condition \eqref{cdt:small_BsnqN}, we know
that the quantity $\|u\|_{L^{\infty}_t (\sC^{-1})}$ is small. Therefore, from  \eqref{es:tg_PdXBsnqN}, \eqref{es:tsH_PdXBsnqND} and \eqref{es:w_PdXBsnqN}, we infer that 
\begin{multline}\label{es:sH_PdXBsnqN}
\sH(t)  \lesssim  \|\d_{X_0} u_0\|_{\sC^{\ep-2}} + \|X_0\|_{\widetilde{\sC}^{\ep}}(\|u_0\|_{B^{\frac{N}{p}-1}_{p,1}}+1) + \|X_0\|_{L^{\infty}} U_p(t) e^{U_p(t)}
  \\
    +\|\d_X \theta\|_{L^1_t (\sC^{\ep-2})}  + \int^t_0 U'_p\,\sH\,dt'    + \int^t_0 \big( U'_p(U_p+1)+ \|(\d_t u,\nabla\Pi,\theta)\|_{B^{\frac{N}{p}-1}_{p,1}} \big) \|X\|_{\widetilde{\sC}^{\ep}} \,dt'.
\end{multline}

Denoting $$\sK(t) := \sH(t)+\| X \|_{L^{\infty}_t (\widetilde{\sC}^{\ep})}$$
and using \eqref{es:X_PdXBsnqN} and \eqref{es:sH_PdXBsnqN}, we obtain that
$$
\displaylines{
\sK(t) \lesssim \big( \|\d_{X_0} u_0\|_{\sC^{\ep-2}} + \|X_0\|_{\widetilde{\sC}^{\ep}}(\|u_0\|_{B^{\frac{N}{p}-1}_{p,1}}+1) + \|X_0\|_{L^{\infty}} U_p(t) e^{CU_p(t)}+ \|\d_X \theta\|_{L^1_t (\sC^{\ep-2})}\big) \hfill\cr\hfill 
 + \int^t_0 \Big(  U'_p(U_p+1)+ \|(\d_t u,\nabla\Pi,\theta)\|_{B^{\frac{N}{p}-1}_{p,1}} \Big) \sK(t')\, dt'.}
$$
In order to bound $\d_X\theta,$ it suffices to remember that
$$
\d_t\d_X\theta+u\cdot\nabla\d_X\theta=0.
$$
 As $\frac{N}{p}+\ep > 1$, applying  Proposition \ref{prop:trans_diff} gives 
\begin{equation*}
\|\d_X \theta\|_{L^{\infty}_t (\sC^{\ep-2})} \lesssim \|\d_{X_0} \theta_0\|_{ \sC^{\ep-2}}e^{CU_p(t)}.
\end{equation*}  
Furthermore, using embedding and   Vishik type estimate \eqref{es:vishik_BsnqN} for $\theta$ yields
\begin{equation*}
\|\theta\|_{L^{\infty}_t (B^{\frac{N}{p}-1}_{p,1})} \lesssim\|\theta\|_{\widetilde{L}^{\infty}_t (B^{0}_{N,1})} \lesssim \|\theta_0\|_{B^{0}_{N,1}}(1+U_p(t)) .
\end{equation*}
Applying  Gronwall lemma to $\sK(t)$ and mingling with the above bounds for $(\d_X \theta, \theta)$ 
and with the bounds of Theorem \ref{thm:wellposed_BsnqN} for $U_p$ and $(\d_tu,\nabla\Pi)$ yield
\begin{multline*}
\sK(t) \leq C \big( \|\d_{X_0} u_0\|_{\sC^{\ep-2}} + \|X_0\|_{\widetilde{\sC}^{\ep}}(\|u_0\|_{B^{\frac{N}{p}-1}_{p,1}}+1) + \|X_0\|_{L^{\infty}} + \|\d_{X_0} \theta_0\|_{\sC^{\ep-2}}\big) \cr\hfill \exp \Big(C\big((A^2(t)+1)(t+1)(\|\theta_0\|_{B^{0}_{N,1}}+1)\big)\Big)
\leq C_{0} \exp \big( \exp ( C_{0}t) \big), 
\end{multline*}
where  $A(t)$ has been  defined in Theorem \ref{thm:wellposed_BsnqN}. Then we deduce from \eqref{es:dXu_TXu_BsnqN1}
and   \eqref{es:dXu_TXu_BsnqN2} that
\begin{equation*}
\|\d_X u\|_{L^{\infty}_t (\sC^{\ep-2})}+\|\d_X u\|_{\widetilde {L}^{1}_t(\sC^{\ep})} \lesssim (1+A(t)) \sK(t).
\end{equation*}
Finally, in order to bound $\d_X\nabla\Pi,$ we use the identity
$$
\d_X\nabla\Pi-\nabla\cT_X\Pi=\div R(X,\nabla\Pi)-R(\div X,\nabla\Pi)+T_{\nabla^2\Pi}\cdot X-\cT_{\nabla X}\Pi,
$$
from which we deduce that
$$
\|\d_X\nabla\Pi-\nabla\cT_X\Pi\|_{\wt{L}_t^1(\sC^{\ep-2})}\lesssim \|X\|_{L_t^\infty(\wt\sC^\ep)}\|\nabla\Pi\|_{L^1_t(B^{\frac Np-1}_{p,1})}.$$
Now, the fact that  $\nabla \cT_X \Pi$ may be estimated in  $\big( \wt L^1_{loc}(\R_+;\sC^{\ep-2})\big)^N$ is a consequence of 
\begin{equation*}
\nabla \cT_X \Pi= \mathbb{Q} \widetilde{g} 
\end{equation*} 
and of \eqref{es:tg_PdXBsnqN} once noticed that  the projector $\mathbb{Q}$  on gradient-like vector fields
is a $0$-th order Fourier multiplier
and that the Fourier transform of $\nabla \cT_X \Pi$ is supported away from~$0.$
This completes the proof of a priori estimates for striated regularity in Theorem \ref{thm:dXBsnqN}.

%%%%%%%%%%%%%%%%%%%%%%%%%%%%%%%%%%%%%%%

\subsection{End of the proof of Theorem \ref{thm:dXBsnqN}}
Like in 2-D case, we start by smoothing out the initial data:
$$
(\theta^n_0, u^n_0) := (S_n \theta_0, S_n u_0).
$$
Since $S_n$ is a (uniformly) bounded operator  on all Besov spaces $B^s_{q,r},$  Lebesgue spaces $L^q$ and weak  Lebesgue spaces $L^{q,\infty}$ with $1< q <\infty,$  we can apply Theorem \ref{thm:wellposed_BsnqN} and obtain
 for all $n\in\N$  a unique global solution $(\theta^n , u^n, \nabla \Pi^n)$ 
which fulfills Inequalities \eqref{es:lorentz_BsnqN} to \eqref{es:uPi_BsnqN} with right-hand sides independent of $n$.
In particular, sequence $(\theta^n,u^n,\nabla\Pi^n)_{n\in\N}$ is bounded in the space $F_p^0(\R^N)$ of  Subsection \ref{sect:complt_thm_dXBsnq2}.

\medbreak

If we consider as in the 2-D case, System \eqref{eq:BsnqmnN} satisfied by
 $$(\delta \theta^m_n, \delta u^m_n, \delta \Pi^m_n) := (\theta^m-\theta^n, u^m-u^n, \Pi^m-\Pi^n),$$
and combine estimates for transport and transport-diffusion equations, and product laws like e.g
$$
\begin{aligned}
\|\div (\delta u^m_n  \theta^n)\|_{B^{\frac{N}{p}-2}_{p,1}}  &\lesssim \|\delta u^m_n \|_{B^{\frac{N}{p}}_{p,1}} \|\theta^n\|_{B^{\frac{N}{p}-1}_{p,1}}\cr \mbox{and}\quad\|\div(\delta u^m_n \otimes  u^n)\|_{B^{\frac{N}{p}-2}_{p,1}}  &\lesssim \|\delta u^m_n \|_{B^{\frac{N}{p}-2}_{p,1}} \|u^n\|_{B^{\frac{N}{p}+1}_{P,1}},\end{aligned} 
$$
then we can easily prove that the sequence $(\theta^n, u^n, \nabla \Pi^n)_{n\in\N}$ strongly converges
in  $F^{-1}_p(\R^N)$  to some limit $(\theta,u,\nabla \Pi).$ By uniform estimates and interpolation, we see that
in addition convergences holds true  in $F^{-\eta}_p(\R^N),$ for any $\eta>0,$ and  functional analysis arguments
similar to those of Subsection \ref{sect:complt_thm_dXBsnq2} 
 allow to show that $(\theta,u,\nabla\Pi)$ is in indeed a solution to $(B_{1,N})$
 that belongs to $F^0_p.$ Moreover, $\theta \in \widetilde{L}^{\infty}_{loc}(\R_{+}; B^{0}_{N,1}),$ owing  to \eqref{es:theta_BsnqN}.
  \medbreak
  Next, to establish conservation of striated regularity, we smooth out $X_0$ into $X_0^n:=S_{n}X_0$ and define $ X^n(t,x) := (\d_{X_0^n} \psi_{u^n})(\psi^{-1}_{u^n}(t,x)).$ By adapting the arguments of the 2-D case, it is easy to make all the bounds obtained in the previous step independent of $n,$ and to conclude to the full statement of Theorem  \ref{thm:dXBsnqN}.
The details are left to the reader.

%%%%%%%%%%%%%%%%%%%%%%%%%%%%%%%%%%%%%%

\section{The temperature patch problem with H\"older regularity}\label{sect:5}

This section is devoted to solving the temperature patch problem, through Corollaries \ref{cor:dXBsnq2} and \ref{cor:striated_BsnqN}. We start with general considerations that will be useful both in the 2-D and in the N-D cases.

Because  $\cD_0$ is a simply connected bounded $C^{1,\ep}$ domain of $\R^N$ (see Definition above Corollary \ref{cor:striated_BsnqN}), it  is orientable (as  $\d \cD_0$  is compact), and  one can 
adopt the so-called  level-set characterization (see Chap. 2 of \cite{doC1976} and references therein).
 More precisely,  for any (bounded) open neighborhood $V_0$ of $\overline{\cD}_0,$ there exists some open set $W_0$ with $\overline{\cD}_0\subset W_0\subset V_0$ and a  
 function $f_0 \in C^{1,\ep}(\R^N ; \R)$ such that  $\d \cD_0 = f_0^{-1}(\{0\})\cap W_0,$ $\nabla f_0$ is supported in $\overline W_0$  and does not vanish on   $\d \cD_0$. The key to that global characterization is the \emph{tubular neighborhood theorem.} 
 In the case $N=2,$ we fix   $V_0$ so  that in addition $V_0\cap\overline{\cD_0^\star}=\emptyset$ 
 (see  Corollary \ref{cor:dXBsnq2}) while if $N\geq3$, we assume that $\overline{V_0} \cap (\R^N\setminus J_0)=\emptyset.$ 
 Finally, we also  fix some smooth cut-off function $\chi_0$ with compact support in $V_0$, and  value $1$ on $W_0.$
    \medbreak
As for the classical vortex patch problem for the incompressible Euler equations, 
the result will come up as a consequence of persistence of striated regularity. 
Indeed,  at time $t$  the boundary of  $\cD_t = \psi_u (t, \cD_0)$ is the level set of the function $f(t,\cdot):= f_0 \big(\psi_u^{-1}(t,\cdot)\big)$ satisfying
\begin{equation}
\left\{
\begin{array}{l}
\d_t f +u\cdot \nabla f = 0, \\
 f|_{t=0}=f_0. \\
\end{array}
\right.
\end{equation}
In order to control the regularity in $C^{1,\ep}$ of $\d \cD_t,$ 
it is thus sufficient to show that $\nabla f$ remains in  $(C^{0,\ep})^N$ for all time. In the 2-D case,  the proof  is rather direct because 
it will be possible to apply directly Theorem \ref{thm:dXBsnq2} with the (divergence free) 
 vector field $X_0 := \nabla^{\perp} f_0$ provided $u_0$ satisfies suitable striated regularity properties
 with respect to $X_0.$    
  In the N-D case, more vector fields will be needed, and they will not be divergence free any longer.
\medbreak
Before going to the proof of our corollaries, we would like to point out  two important properties.  The first 
one is that  the characteristic function of any  bounded $C^1$ bounded simply connected domain  $\cD$ of  $\R^N$ for $N \geq 2$ (see \cite{DanM2012} and references therein) satisfies
\begin{equation}\label{obs1:cor_dXBsnq}
\mathds{1}_{\mathcal{D}}\in B^{\frac{1}{q}}_{q,\infty}(\R^N) \hookrightarrow B^{\frac{N}{p}-1}_{p,1}(\R^N),  ~~~ \forall~~~ (N-1)< q\leq p \leq \infty.
\end{equation} 
\medbreak
The second observation concerns $L^N$ boundedness of velocity field $u_0$.  

In  the 2-D case, having  the vorticity compactly supported,  in $L^r$ for some $r>1,$ 
and \emph{mean free} ensures that $u_0\in \big( L^2(\R^2) \big)^2$  (see e.g.  Chap.3 in \cite{MajB2002}). 

 In dimension $N\geq3,$ the $L^N$ type information is an easy consequence of  Biot-Savart law and Hardy-Littlewood-Sobolev inequality. Indeed, we have
\begin{equation}\label{obs2:cor_dXBsnq}
\|u_0\|_{L^N(\R^N)} \lesssim \|~ |\cdot |^{1-N} \star \Omega_0 \|_{L^N(\R^N)} \lesssim \|\Omega_0\|_{L^{\frac{N}{2}}(\R^N)}.
\end{equation}
 In what follows, whenever $Y$ is a (tangential) vector field on $\R^N,$ we shall denote by $\widetilde Y$ its restriction along the surface $\d \cD_t.$ 

%%%%%%%%%%%%%%%%%%%%%%%%%%%%%%%%%%%%%%%%

\subsection{The 2-D case} 
This short paragraph is devoted to proving  Corollary \ref{cor:dXBsnq2}.
Of course, one can assume with no loss of generality that the exponent $r$ is in $]1,\frac{2}{2-\ep}[$
so that we have $\omega_0 \in B^{\frac{2}{q}-2}_{q,1}(\R^2)$ for some $1<q<\frac2{2-\ep}\cdotp$
Thus, remembering that $u_0$ is in $\bigl(L^2(\R^2)\bigr)^2$ (as explained above), and 
using \eqref{obs1:cor_dXBsnq}, we see that 
 both   $\theta_0$ and $u_0$ fulfill the regularity assumptions of Theorem  \ref{thm:dXBsnq2}.

Next, we observe that  the boundary  $\d \cD_0$  of the initial patch  is just some Jordan curve of the plane.
Setting $X_0 := \nabla^{\perp} f_0$ provides us with   a parametrization $\gamma_0$ of the curve $\d \cD_0$ through the following ODE:
\begin{equation*}
\left\{
\begin{array}{l}
\d_\sigma \gamma_0 (\sigma) =  \widetilde{X}_0\big(\gamma_0 (\sigma)\big), ~~~\forall \sigma \in \mathbb{S}^1,\\
 \gamma_0(\sigma_0)= x_0 \in \d \cD_0. \\
\end{array}
\right.
\end{equation*}
  It is obvious that  $\d_{X_0} \theta_0=\div(X_0\theta_0)\equiv0$ in the sense of distributions. 
  Similarly, because $X_0\equiv0$ on some neighborhood of $\overline{\cD_0^\star},$
 we have  $\div(X_0 \omega_0) = \d_{X_0} \omega_0 \equiv0.$ Hence Theorem \ref{thm:dXBsnq2} applies.
 \smallbreak
 In order to conclude the proof of Corollary \ref{cor:dXBsnq2}, we observe that
  a parametrization for $\d \cD_t$ is given by $\gamma_t(\sigma):= \psi_u(t, \gamma_0\big(\sigma)\big)$ and that, obviously, \begin{equation*}
\left\{
\begin{array}{l}
\d_\sigma \gamma_t (\sigma) =  \widetilde{X}_t\big(\gamma_t (\sigma)\big), ~~~\forall \sigma \in \mathbb{S}^1,\\
 \gamma_t(\sigma_0)= \psi_u (t,x_0) \in \d \cD_t, \\
\end{array}
\right.\quad\hbox{with }\ X_t(x):= (\d_{X_0} \psi_u)\big(\psi^{-1}_{u}(t,x)\big).
\end{equation*}
 As pointed out before, $X_t$ satisfies Equation \eqref{eq:X}, from which we can deduce that  $X_t$ is in $\bigl(C^{0,\ep}(\R^2)\bigr)^2$ for all $t\geq0,$ and thus $\gamma_t \in C^{1,\ep}(\mathbb{S}^1; \R^2)$. 

%%%%%%%%%%%%%%%%%%%%%%

\subsection{The N-D case}

If we take data $\theta_0$ and $u_0$ fulfilling the hypotheses of Corollary \ref{cor:striated_BsnqN}
then  \eqref{obs1:cor_dXBsnq} and standard embedding ensure that
\begin{equation*} 
\theta_0 \in  B^{0}_{N,1}(\R^N)\quad\hbox{and}\quad
 (\Omega_0)^i_j \in B^{\frac{1}{p}}_{p,\infty}(\R^N) \hookrightarrow B^{\frac{N}{p}-2}_{p,1}(\R^N) ~~\mbox{ as long as }~ p > \frac{N-1}{2}\cdotp
\end{equation*}
Since in addition $u_0^i \in L^N(\R^N)$ (see \eqref{obs2:cor_dXBsnq}), we have 
$u_0^i \in B^{\frac{N}{p}-1}_{p,1}(\R^N)$ as may be seen through the following decomposition
$$
u_0^i=\Delta_{-1}u_0^i-\sum_{j}({\rm Id}-\Delta_{-1})(-\Delta)^{-1}\d_j(\Omega_0)^i_j,
$$
and the fact that $({\rm Id}-\Delta_{-1})(-\Delta)^{-1}\d_j$ is a homogeneous multiplier of degree $-1$
away from a neighborhood of the origin. Indeed, we have
\begin{equation}\label{es:u0_dXBnsq}
\|u_0\|_{B^{\frac{N}{p}-1}_{p,1}} \lesssim \|\Omega_0\|_{L^{\frac{N}{2}}} + \|\Omega_0\|_{B^{\frac{1}{p}}_{p,1}}, ~~~\forall~p >\frac{N-1}{2}.
\end{equation}
Therefore, if  $m_1$ and $m_2$  in Corollary \ref{cor:striated_BsnqN} are sufficiently small, then we may apply the first part of Theorem \ref{thm:dXBsnqN} and get a unique global solution $(\theta, u)$ for System \eqref{eq:BsnqN}, satisfying \eqref{es:lorentz_BsnqN}, 
\eqref{es:Up_BsnqN}, \eqref{es:theta_BsnqN} and \eqref{es:uPi_BsnqN}.
  \medbreak
  
  In order to propagate the regularity of the temperature patch,  we shall borrow the ideas of  \cite{Dan1999, GS-R1995} pertaining to the standard  vortex patch problem. Firstly, we claim that it is sufficient to verify that $\d_{X_0} \Omega_0$ has regularity of $\sC^{\ep-3}$ in order to gain $\d_{X_0} u_0 \in  \bigl( \sC^{\ep-2}(\R^N) \bigr)^N,$ for any vector field $X_0$ in $\bigl(\widetilde{\sC}^\ep(\R^N) \bigr)^N.$ Indeed, Lemma \ref{lemma:ce3} yields for $\frac{N}{p}+\ep > 1$
\begin{equation*}
\| \d_{X_0} u_0 \|_{\sC^{\ep-2}} \lesssim \|\cT_{X_0}u_0\|_{\sC^{\ep-2}} + \|X_0\|_{\widetilde{\sC}^\ep} \|u_0\|_{B^{\frac{N}{p}-1}_{p,1}}.
\end{equation*}
Then by Biot-Savart Law, we have the following identity
\begin{eqnarray*}
\cT_{X_0}u_0^i & =& \sum_{j} T_{X_0^k}\d_k \Delta^{-1}\d_j (\Omega_0)^i_j \\
               & =& \sum_{j} \bigl(\Delta^{-1}\d_j \psi_{N_0}(D)\bigr) \cT_{X_0} (\Omega_0)^i_j + [T_{X_0^k}, \Delta^{-1} \d_j \psi_{N_0}(D)] \d_k (\Omega_0)^i_j,
\end{eqnarray*} 

where $\psi_{N_0}$ is some suitable (smooth) cut-off function away from $0.$ 
From the above identity and  Lemma \ref{lemma:ce1}, we deduce  that
$$
\|\cT_{X_0} u_0\|_{\sC^{\ep-2}} \lesssim \|\cT_{X_0}\Omega_0\|_{\sC^{\ep-3}}+
\|X_0\|_{\widetilde{\sC}^{\ep}}\|\Omega_0\|_{B^{\frac{1}{q}}_{q,\infty}}\quad\hbox{for }\ q\geq N-1,
$$
whence, thanks to  Lemma \ref{lemma:ce3} and \eqref{es:u0_dXBnsq},
\begin{equation}\label{es:dXu_PdXOmega}
\|\d_{X_0} u_0\|_{\sC^{\ep-2}} \lesssim \|X_0\|_{\widetilde{\sC}^{\ep}}(\|\Omega_0\|_{L^{\frac{N}{2}}} + \|\Omega_0\|_{B^{\frac{1}{q}}_{q,\infty}} )+ \|\d_{X_0} \Omega_0\|_{\sC^{\ep-3}}.
\end{equation}
 \medbreak
  
  Now it is easy construct the first family of vector fields $\{ E_j \}_{j=1}^N$ according to \eqref{es:dXu_PdXOmega} and Theorem \ref{thm:dXBsnqN}  as follows. Let  $\{e_1,\cdots,e_N\}$  be the canonical basis of $\R^N$ and assume (with no loss of generality) that the cut-off function $\chi_0$
defined at the beginning of this section is supported in the domain $J_0.$
Define $E_j$ be the solution to 
\begin{equation*}
\left\{
\begin{array}{l}
\d_t E_j +u\cdot \nabla E_j = \d_{E_j} u, \\
 E_j|_{t=0}= E_{j,0}:=\chi_0 \, e_j. \\
\end{array}
\right.
\end{equation*}
Because $\nabla\theta_0\in \big( B^{\frac{1}{q}-1}_{q,\infty}(\R^N)\big)^N$ for all $q\in[1,+\infty],$
we have $\nabla \theta_0 \in \big(\sC^{\ep-2}(\R^N)\big)^N$ by embedding (just take $q\geq\frac{N-1}{1-\ep}$).
Of course, as $ E_{j,0}$ is smooth and compactly supported, and as any nonhomogeneous Besov space is stable by multiplication by functions in $\cC^\infty_c,$ one can conclude that  $\d_{E_{j,0}} \theta_0 \in \sC^{\ep-2}(\R^N).$ Taking $X_0=E_{j,0}$ in \eqref{es:dXu_PdXOmega}, the last term $\d_{E_{j,0}} \Omega_0$ on the r.h.s. vanishes by the definition of $J_0$. Hence   Theorem \ref{thm:dXBsnqN} applies and  one can conclude that   $E_j \in \big(L^\infty _{loc}(\R_+ ;\widetilde{\sC}^{\ep})\big)^N$ for all $j=1,\cdots,N.$
\medbreak
  
  Next, let us introduce some notations before our construction of second family of vector fields $\{Y_{\lambda}\}$. For any family $\cP= (P_1,...,P_{N-1})$ of $N-1$ vectors of $\R^N,$ we define $\wedge\cP \equiv P_1 \wedge \cdots \wedge P_{N-1}$ to be  the unique vector of  $\R^N$ satisfying
\begin{equation*}
\wedge\cP\cdot Q = \det (P_1,...,P_{N-1},Q), ~~~\forall Q \in \R^N.
\end{equation*}
Assume that $\cX=(X_1,\cdots,X_{N-1})$ where the time dependent vector fields $X_j$ satisfy System \eqref{eq:X}. Thanks to $\div u=0$, the vector field $\wedge\cX$ satisfies the equation,
\begin{equation}\label{eq:wedgeX}
\d_t \wedge\!\cX + u \cdot \nabla \wedge\!\cX = - (\nabla u)^{tr} \cdot  (\wedge\cX).
\end{equation}
Set $\displaystyle\Lambda = \{ \lambda = (\lambda_1,...,\lambda_{N-1}): 1 \leq \lambda_1<\cdots<\lambda_{N-1}\leq N \}.$
For any $\lambda \in \Lambda$, we define $Y_{\lambda}$ to be the solution~of 
\begin{equation*}
\left\{
\begin{array}{l}
\d_t Y_\lambda +u\cdot \nabla Y_\lambda = \d_{Y_\lambda} u, \\
 Y_{\lambda}|_{t=0}=e_{\lambda_1}\wedge \cdots \wedge e_{\lambda_{N-2}} \wedge \nabla f_0. \\
\end{array}
\right.
\end{equation*}
Extend $\lambda$ to  a permutation  $\bar{\lambda}:=( \lambda_1,\cdots,\lambda_{N-1},  \lambda_N)$ 
of $\{1,\cdots,N\},$ with  signature  $\tau(\bar{\lambda}).$
By definition of $\bar\lambda,$  we have 
\begin{equation*}
Y_{\lambda,0} \equiv  Y_{\lambda}|_{t=0} = \tau(\bar{\lambda})(\d_{\lambda_{N-1}}f_0\, e_N - \d_{\lambda_N}f_0\, e_{N-1}),
\end{equation*}
which is divergence free in $\big(\sC^{\ep}(\R^N)\big)^N.$ Because 
$\widetilde{Y}_{\lambda,0}$ is a (non-degenerate)  tangential vector field along $\d \cD_0,$ 
the function $\d_{Y_{\lambda,0}} \theta_0$ vanishes.  In addition, from the fact
 $$\mbox{Supp} \nabla (\Omega_0)^i_j \cap \mbox{Supp} \chi_0 = \emptyset  ~~\mbox{  and  }~~ \nabla f_0 = \chi_0 \nabla f_0,$$ 
we gather that $\d_{Y_{\lambda,0}} \Omega_0$ vanishes for any $\lambda.$ Applying  \eqref{es:dXu_PdXOmega} 
with  $X_0 = Y_{\lambda,0}$ and Theorem \ref{thm:dXBsnqN}, we get $Y_{\lambda} \in \big( L^\infty _{loc}(\R_+;{\sC}^{\ep})\big)^N$.
Note that in fact  $Y_{\lambda} \in \big( L^\infty _{loc}(\R_+;\wt{\sC}^{\ep})\big)^N$ as all those vector fields are divergence free.
\medbreak

Finally, putting together what we proved hitherto, we deduce that the vector fields $W_\lambda$ defined by  
\begin{equation*}
W_{\lambda} := E_{\lambda_1}\wedge \cdots \wedge E_{\lambda_{N-2}} \wedge Y_{\lambda},
\end{equation*}
are in $\big(L^\infty _{loc}(\R_+; \sC^{\ep})\big)^N.$
Furthermore, according to the definition of $E_j$ and $Y_{\lambda},$ they satisfy Equation \eqref{eq:wedgeX}.
The expression of $Y_{\lambda,0}$ and the fact that  $\nabla f_0 = \chi_0 \nabla f_0$ imply that
\begin{equation*}
W_{\lambda,0} \equiv  W_{\lambda}|_{t=0} = - (\d_{\lambda_{N-1}}f_0 e_{N-1} + \d_{\lambda_N}f_0 e_{N}),
\end{equation*}
and summing up over $\lambda\in\Lambda$ thus yields
\begin{equation*}
\sum_{\lambda \in \Lambda} W_{\lambda,0} = -(N-1) \nabla f_0.
\end{equation*}
Therefore
\begin{equation*}
W:= -\frac{1}{N-1} \sum_{\lambda \in \Lambda} W_{\lambda}   \in \big( L^\infty _{loc}(\R_+;\sC^{\ep})\big)^N,
\end{equation*}
coincides with $\nabla f$ thanks to \eqref{eq:wedgeX} and  to the uniqueness of solutions for the equation satisfied by $\nabla f $, namely,
\begin{equation*}
\left\{
\begin{array}{l}
\d_t \nabla f +u\cdot \nabla (\nabla f) =  - (\nabla u)^{tr} \nabla f, \\
 \nabla f|_{t=0}=\nabla f_0. \\
\end{array}
\right.
\end{equation*}
This completes the proof of the conservation of $C^{1,\ep}$ regularity for domain $\cD_t$ in 
dimension $N\geq3.$

%%%%%%%%%%%%%%%%%%%%%%%%%%%%%%%%%%%%%%%%

\begin{appendix}

\section{Commutator Estimates}

We here prove some  commutator estimates that were needed in the previous
sections. They strongly rely on continuity results in Besov spaces for the paraproduct and remainder operators, and on  the following classical result (see e.g.  \cite{BCD2011}, Chap. 2).
\begin{lemma}\label{lemma:ce1} 
Let $A:\R^N\to\R$ be a smooth function, homogeneous of degree $m$ away from  a neighborhood of $0$. Let $(\ep,s,p,r,p_1,p_2) \in ]0,1[ \times \R \times [1, \infty]^4$ with $\frac{1}{p}=\frac{1}{p_1}+\frac{1}{p_2}\cdotp$ Then there exists a constant $C$, depending only on $s, \ep, N$ and $A$ such that,
$$\|[T_g, A(D)]u\|_{B^{s-m+\ep}_{p,r}} \leq C \|\nabla g\|_{B^{\ep-1}_{p_1,\infty}}\|u\|_{B^{s}_{p_2,r}}.$$
\end{lemma}

\begin{rmk} A similar inequality  holds for time-dependent distributions, as may be seen by   following 
the proof of  Lemma \ref{lemma:ce1}, treating the time as a parameter, and applying (time) H\"older inequality 
when appropriate. 
For example, for any  $(\rho,\rho_1,\rho_2) \in [1,\infty]^3$ with $\frac{1}{\rho}=\frac{1}{\rho_1}+\frac{1}{\rho_2},$
and $(\ep,s,p,r,p_1,p_2)$ as above, we  have
\begin{equation}\label{es:ce2} 
\|[T_g, A(D)]u\|_{\widetilde{L}^{\rho}_T (B^{s-m+\ep}_{p,r})} \leq C \|\nabla g\|_{\widetilde{L}^{\rho_1}_{T} (B^{\ep-1}_{p_1,\infty})}\|u\|_{\widetilde{L}^{\rho_2}_T (B^{s}_{p_2,r})}.
\end{equation} 
\end{rmk}
\begin{lemma}\label{lemma:ce3}
Let the vector field $X$ be  in $\big(\sC^{\ep}(\R^N)\big)^N$ for some $\ep\in]0,1[,$  and $f$ be in 
$B^{s}_{p,r}$ with $(s,p,r) \in ]-\infty, 1+\frac{N}{p}[ \times [1,\infty]^2$. Then we have:
  \begin{enumerate}
  \item If in addition  $s + \ep >1$ or $\{ s+\ep=1, r=1 \}$, then
  $$\|\cT_X f - \d_X f\|_{\sC^{s+\ep-\frac{N}{p}-1}} \lesssim \|X\|_{\sC^{\ep}} \|\nabla f\|_{B^{s-1}_{p,r}}.$$
  The previous estimate holds in the case $s =1+\frac{N}{p},$ if replacing $\|\nabla f\|_{B^{\frac{N}{p}}_{p,r}}$ by $\|\nabla f\|_{B^{\frac Np}_{p,\infty}\cap L^{\infty}}$.
  \item If in addition  $\div X \in \sC^\ep(\R^N)$ then  for $s+\ep>0$ or $\{ s+\ep=0, r=1\}$,  we have 
  $$\|\cT_X f - \d_X f\|_{\sC^{s+\ep-\frac{N}{p}-1}} \lesssim \|X\|_{{\widetilde{\sC}}^\ep}  \|f\|_{B^{s}_{p,r}}.$$

  \end{enumerate} 
\begin{proof}
One may start from the Bony decomposition as follows: 
\begin{align*}
\d_X f &= \cT_X f + T_{\d_k f} X^k +R(X^k,\d_k f)\\
  &= \cT_X f + T_{\d_k f} X^k +\d_k R(X^k, f)-R(\div X, f).
\end{align*} 
Then applying standard  continuity results for  paraproduct and remainder operators 
(see e.g.  \cite{BCD2011}, Chap. 2)  yields the desired inequalities.
 \end{proof}
\end{lemma}
If the integer $N_0$ in the definition of Bony's paraproduct and remainder  is large enough (for instance $N_0 =4$ does), 
then the following fundamental lemma holds.
\begin{lemma}\label{lemma:key_Bnsq}
Let $(\ep,s,s_k,p, p_k, r,r_k) \in ]0,1[ \times \R^2 \times [1, \infty]^4$ for $k=1,2$ satisfying
\begin{equation*}
\frac{1}{p} = \frac{1}{p_1} + \frac{1}{p_2}\ \hbox{ and }\ \frac{1}{r} = \frac{1}{r_1} + \frac{1}{r_2}\cdotp
\end{equation*} 

\begin{enumerate}
\item If $s_2  < 0$ then we have 
\begin{equation}
\|\cT_X T_g f - T_g \cT_X f - T_{ \cT_X g}f \|_{B^{s_1+s_2+\ep-1}_{p,r}} \leq C \|X\|_{\sC^{\ep}} \|f\|_{B^{s_1}_{p_1,r_1}}\|g\|_{B^{s_2}_{\infty,r_2}}.
\end{equation}
The above inequality still  holds in the limit case $(s_2, r_2)=(0, \infty),$ if one  replaces $\|g\|_{B^{0}_{\infty,\infty}}$ by $\|g\|_{L^\infty}$.
\item If $s_1+s_2+\ep-1 >0$ then  we have 
\begin{equation*}
\|\cT_X R(f,g)-R(\cT_X f, g)-R(f, \cT_X g)\|_{B^{s_1+s_2+\ep-1}_{p,r}} \leq C \|X\|_{\sC^\ep} \|f\|_{B^{s_1}_{p_1,r_1}}\|g\|_{B^{s_2}_{p_2,r_2}}.
\end{equation*} 
The above inequality still holds in the limit case $s_1+s_2+\ep -1 =0,$ $r=\infty$ and  $\frac{1}{r_1} + \frac{1}{r_2}=1.$
\end{enumerate}

\end{lemma}
\begin{proof}
The key to our proof is a generalized Leibniz formula for the para-vector field operators which was derived by J.-Y. Chemin in \cite{Che1988}. Define the following Fourier multipliers for $k \in {1,...,N}$
\begin{equation*}\Delta_{k,j} := 
\begin{cases} 
\varphi_{k}(2^{-j}D) & j \geq 0,\quad\ \,\hbox{ with }\ \varphi_{k}(\xi):= i \xi_k \varphi(\xi),\\
\chi_{k}(D) & j=-1,\quad\hbox{with }\ \chi_{k}(\xi):= i \xi_k \chi(\xi),\\
0 & j \leq -2.
\end{cases}
\end{equation*}
Then we have
\begin{align*}
 \cT_X T_g f & = \sum_{j \geq -1} (S_{j-N_0} g \cT_{X}\Delta_{j} f + \Delta_j f \cT_X S_{j-N_0} g) + \sum_{j\geq -1}(T_{1,j} + T_{2,j})\\
 &= T_g \cT_X f +T_{ \cT_X g}f + \sum_{\genfrac{}{}{0pt}{}{j \geq -1}{\alpha=1,\ldots,4}} T_{\alpha, j},
\end{align*}
where 
\begin{align*}
T_{1,j} &:=  \sum_{\genfrac{}{}{0pt}{}{j \leq j' \leq j+1}{j-N_0-1 \leq j''\leq j'-N_0-1}} 2^{j'}  \Delta_{j''}X^k \big( \Delta_{k,j'} (\Delta_j f S_{j-N_0} g)-\Delta_{k,j'}\Delta_j f S_{j-N_0}g \big),\\
T_{2,j} &:= \sum_{\genfrac{}{}{0pt}{}{j'\leq j- 2}{j'-N_0 \leq j''\leq j-N_0 -2}} 2^{j'}  \Delta_{j''}X^k   (\Delta_j f)  \Delta_{k,j'} S_{j-N_0}  g,\\
T_{3,j} & := S_{j-N_0} g [T_{X^k}, \Delta_j] \d_k f,\\
T_{4,j} & := \Delta_j f [T_{X^k}, S_{j-N_0}] \d_k g.
\end{align*}
Bounding $T_{1,j}$ and $T_{2,j}$ is straightforward : just  use the definition 
of Besov norms. 
Lemma 2.100  of  \cite{BCD2011} allows to bound $T_{3,j}$ and $T_{4,j}$
provided  $\ep<1$. We end up with the desired inequality.
\medbreak
In order to prove the second item, let us set $A_{j,j'} := [j-N_0-1, j'-N_0-1] \cup [j'-N_0, j-N_0-2] \cap \mathbb{Z}.$
We thus have
\begin{align*}
 \cT_X R(f, g) & = \sum_{j \geq -1} (\tDelta_{j}g \cT_{X}\Delta_{j} f + \Delta_j f \cT_X \tDelta_{j} g) + \sum_{j\geq -1}(R_{1,j} + R_{2,j})\\
 &= R(\cT_X f, g)+R(f, \cT_X g) + \sum_{\genfrac{}{}{0pt}{}{j \geq -1}{\alpha=1,\ldots,4}} R_{\alpha, j},
\end{align*}
where, denoting $\tDelta_j:=\Delta_{j-N_0}+\cdots+\Delta_{j+N_0},$ 
\begin{align*}
R_{1,j} &:=  \sum_{\genfrac{}{}{0pt}{}{|j'-j| \leq N_0 + 1}{j''\in A_{j,j'}}} {\rm sgn}(j'-j+1) 2^{j'}  \Delta_{j''}X^k \big( \Delta_{k,j'} (\Delta_j f \tDelta_{j} g) - \Delta_j f \Delta_{k,j'} \tDelta_{j}g \big)\\
& \hspace{3cm}+ \sum_{\genfrac{}{}{0pt}{}{j-1 \leq j' \leq j}{j'-N_0 \leq j'' \leq j-N_0}} 2^{j'} \Delta_{j''}X^k(\Delta_{k,j'} \Delta_j f) \tDelta_{j} g,\\
R_{2,j} &:= \sum_{\genfrac{}{}{0pt}{}{j' \leq j-N_0-2}{j'-N_0 \leq j''\leq j-N_0-2}} 2^{j'}  \Delta_{j''}X^k   \Delta_{k,j'} (\Delta_j f \tDelta_{j}  g)   ,\\
R_{3,j} & := \tDelta_j g [T_{X^k}, \Delta_j] \d_k f,\\
R_{4,j} & := \Delta_j f [T_{X^k}, \tDelta_j] \d_k g.
\end{align*}
Here again, bounding $R_{1,j}$ and $R_{2,j}$ follows from the definition of
Besov norms, while Lemma 2.100 of  \cite{BCD2011} allows to bound $R_{3,j}$ and $R_{4,j}.$
\end{proof}

We are in  position to prove our key commutator estimate that
involves the convective derivative and some para-vector field evolving according to \eqref{eq:X}. 
A particular case of that estimate has been proved in \cite{Dan1997}, as it was needed to study the viscous
vortex patch problem.  
\begin{prop}\label{prop:ce_dBnsqN}
Suppose that $(\ep,p)\in (0,1)\times [1,\infty]$ with $\frac{N}{p}+\ep \geq 1 $. Consider a couple of vector fields $(X,v)$
such that   $\div v=0$ 
and $$ (X,v) \in \bigl(L^{\infty}_{loc}(\R_+; \widetilde{\sC}^{\ep}) \bigr)^N \times \big(L^{\infty}_{loc}(\R_+ ; B^{\frac{N}{p}-1}_{p,1})\cap L^{1}_{loc}(\R_+ ;B^{\frac{N}{p}+1}_{p,1})\big)^N,$$ satisfying the following equation
\begin{equation}\label{eq:prop_ce_dBnsqN}
\left\{
\begin{array}{l}
(\d_t+v\cdot \nabla)X = \d_X v,\\
X|_{t=0} = X_0.
\end{array}
\right.
\end{equation}
Then there is a constant $C$ such that:
\begin{multline}\label{es:TX}
 \|[\cT_X, \d_t+v\cdot \nabla ]v\|_{\sC^{\ep-2}} \leq C(\|X\|_{\widetilde{\sC}^\ep} \|v\|_{B^{\frac{N}{p}+1}_{p,1}} \|v\|_{B^{\frac{N}{p}-1}_{p,1}} \\+\|v\|_{\sC^{-1}}\|\cT_{X}v\|_{\sC^\ep}+\|v\|_{B^{\frac{N}{p}+1}_{p,1}}\|\cT_{X}v\|_{\sC^{\ep-2}}).
\end{multline}
 \end{prop}

\begin{proof} 
Because   $\div v=0,$ we may write 
\begin{align*}
[\cT_X, \d_t+v^{\ell} \d_\ell ]v & = -v^\ell\d_\ell T_{X^k}\d_kv- T_{\d_t X^k}\d_kv 
+ T_{X^k}\d_k(v^{\ell}\d_\ell v)\\  
&=- T_{\d_t X^k}\d_kv + \d_\ell\cT_{X}(v^{\ell} v)-\cT_{\d_\ell X}(v^{\ell} v) - v^{\ell} \d_{\ell} \cT_Xv.
\end{align*}
Hence, decomposing  $v^{\ell} v$ according to Bony's decomposition, we discover that  
$$
[\cT_X, \d_t+v^{\ell} \d_\ell ]v = \sum_{\alpha=1}^{\alpha=5} R_{\alpha}
$$
with
\begin{align*}
R_1 &:=- T_{\d_t X^k}\d_kv, &R_2 &:= \d_{\ell}(\cT_X T_{v^{\ell}}v +\cT_X T_{v}v^{\ell} ),\\
R_3 &:= \d_{\ell} \cT_{X} R(v^{\ell},v),   &R_4& :=-\cT_{\d_\ell X}(v^{\ell} v),\\
R_5 &:=- v^{\ell} \d_{\ell} \cT_Xv.
\end{align*}
To complete the proof, it suffices to check that all the terms $R_j$ may be bounded by the r.h.s. of \eqref{es:TX}.

\subsubsection*{$\bullet$ \underline{Bound of $R_1$}:}

From the equation \eqref{eq:prop_ce_dBnsqN}, we have 
\begin{equation*}
R_1 = T_{v\cdot \nabla X^k} \d_k v - T_{\d_X v^k} \d_k v.
\end{equation*}
Hence using standard continuity results for the paraproduct, we deduce that
$$
\|R_1\|_{\sC^{\ep-2}} \lesssim \|\nabla v\|_{\sC^0}\bigl(\|v \cdot \nabla X\|_{\sC^{\ep-2}} + \|\d_X v\|_{\sC^{\ep-2}}\bigr).
$$
The last term may be bounded according to item (ii) of Lemma \ref{lemma:ce3}.  
As for the first term, we use the following decomposition
$$
v\cdot\nabla X=\cT_Xv+T_{\d_kv}X^k+\d_kR(v,X^k)-R(v,\div X),
$$
which allows to get, since  $\ep + \frac{N}{p} \geq 1,$ 
\begin{equation*}
\|R_1\|_{\sC^{\ep-2}}  \lesssim \|\nabla v\|_{L^{\infty}}(\|v\|_{B^{\frac{N}{p}-1}_{p,1}} \|X\|_{\widetilde{\sC}^{\ep}} + \|\cT_X v\|_{\sC^{\ep-2}}).
\end{equation*}
\subsubsection*{$\bullet$  \underline{Bound of $R_2$}:}

Due to Lemma \ref{lemma:key_Bnsq} (i) and continuity of paraproduct operator, we have
\begin{equation*}
\|R_2\|_{\sC^{\ep-2}} \lesssim \|X\|_{\sC^\ep} \|v\|_{\sC^1} \|v\|_{\sC^{-1}} +\|v\|_{\sC^{-1}}\|\cT_{X}v\|_{\sC^\ep}+\|v\|_{\sC^1}\|\cT_{X}v\|_{\sC^{\ep-2}}.
\end{equation*}
\subsubsection*{$\bullet$ \underline{Bound of $R_3$}:}

Applying Lemma \ref{lemma:key_Bnsq} (ii) and continuity of remainder operator under the assumption
 $\frac{N}{p}+\ep \geq 1$ yields
\begin{equation*}
\|R_3\|_{\sC^{\ep-2}} \lesssim \|X\|_{\sC^{\ep}} \|v\|_{B^{\frac{N}{p}+1}_{p,1}}\|v\|_{\sC^{-1}} + \|v\|_{B^{\frac{N}{p}+1}_{p,1}} \|\cT_X v\|_{\sC^{\ep-2}}.
\end{equation*}

\subsubsection*{$\bullet$ \underline{Bound of $R_4$}:}
{}From  Bony decomposition, it is easy to get
\begin{equation*}
\|v^l v\|_{B^{\frac{N}{p}}_{p,1}} \lesssim \|v\|_{\sC^{-1}}\|v\|_{B^{\frac{N}{p}+1}_{p,1}}.
\end{equation*}
Hence
\begin{equation*}
\|R_4\|_{B^{\frac{N}{p}+\ep-2}_{p,1}} \lesssim \|\nabla X\|_{\sC^{\ep-1}}\|v\|_{\sC^{-1}}\|v\|_{B^{\frac{N}{p}+1}_{p,1}}. 
\end{equation*}

\subsubsection*{$\bullet$ \underline{Bound of $R_5$}:}
Applying Bony decomposition and using  that $\div v=0$ and $\frac{N}{p} + \ep \geq 1$ give
\begin{equation*}
\|R_5\|_{\sC^{\ep-2}} \lesssim \|v\|_{\sC^{-1}}\|\cT_{X}v\|_{\sC^\ep}+\|v\|_{B^{\frac{N}{p}+1}_{p,1}}\|\cT_{X}v\|_{\sC^{\ep-2}}.
\end{equation*}
Combining the above  estimates for all $R_{\alpha}$, with $\alpha=1,\dots,5$ yields \eqref{es:TX}. 
\end{proof}
\begin{rmk} Slight modifications of  the above proof allow  to handle time dependent functions with tilde type Besov norms. In particular, one may prove the following estimate for   all  $t>0$: 
$$\displaylines{
 \|[\cT_X, \d_t+v\cdot \nabla ]v\|_{\widetilde{L}^1_t(\sC^{\ep-2})}  \lesssim  \|v\|_{L^{\infty}_t (\sC^{-1})}\|\cT_X v\|_{\widetilde{L}^1_t (\sC^{\ep})}  \hfill\cr\hfill+\int^t_0  \|v\|_{B^{\frac{N}{p}+1}_{p,1}} \|\cT_{X}v\|_{\sC^{\ep-2}}\,dt'
 +  \int_0^t  \|X\|_{\widetilde{\sC}^{\ep}} \|v\|_{B^{\frac{N}{p}+1}_{p,1}} \|v\|_{B^{\frac{N}{p}-1}_{p,1}}\,dt'.}
 $$
\end{rmk}
\end{appendix}
\section*{Acknowledgement}
The first author is partially supported by  ANR-15-CE40-0011.
The second author is  supported by grant of \emph{R\'eseau de Recherche Doctoral en Math\'ematiques de 
l'Ile de France} (RDM-IdF).

\end{document}